\documentclass[reqno,12pt]{amsart}

\usepackage{}
 \usepackage{mathrsfs}
  \usepackage{tensor}
\usepackage{amsfonts, amssymb,amsmath, amsthm, amsxtra, latexsym, amscd}
\usepackage[latin1]{inputenc}
\usepackage{graphicx}
\usepackage{bm}

\theoremstyle{definition}
\newtheorem*{thm*}{Theorem}

\newtheorem{thm}{Theorem}[section]
\newtheorem{cor}[thm]{Corollary}
\newtheorem{lem}[thm]{Lemma}
\newtheorem{prop}[thm]{Proposition}
\newtheorem{exam}[thm]{Example}
\newtheorem{rem}[thm]{Remark}
\newtheorem{defi}[thm]{Definition}
\numberwithin{equation}{section}

\bigskip

\def\refn#1.#2{\expandafter\def\csname#1\endcsname{[#2]}}
\def\refnr#1.{\csname#1\endcsname}

\begin{document}

\baselineskip  1.2pc

\title[  the   compactness  of Bergman-type  operators ]
{ On the  compactness  of Bergman-type  integral  operators }
\author[L. Ding]{Lijia Ding }
\address{School of Mathematical Sciences,
Peking University, Beijing, 100086, P. R. China}
\email{ljding@pku.edu.cn}
\author[J. Fan]{Junmei Fan}
\address{Department of Mathematics Sciences, Dalian University of Technology, Dalian, Liaoning, 116024, P. R. China}
\email{junmeifan666@mail.dlut.edu.cn}
\thanks{The first author was partially supported by PCPSF (2020T130016).}
\begin{abstract}

Bergman-type integral operators  are classical operators  in complex analysis and operator theory.  Recently,  the first author and his collaborator \cite{DiW} completely characterized the $L^p$-$L^q$ boundedness of Bergman-type integral operators  $K_\alpha,K_\alpha^+$ and  the $L^p$-$L^q$ compactness of $K_\alpha$  on the unit ball.  In this paper, we will  use a substantially new method to  completely characterize the $L^p$-$L^q$ compactness of $K_\alpha^+,$  but also prove that the $L^p$-$L^q$ compactness of operators $K_\alpha,K_\alpha^+$ is  in fact equivalent.  Moreover, we completely characterize Schatten class and Macaev class  Bergman-type integral operator $K_\alpha$ on   $L^2$ space and Bergman space via  inequalities related to the dimension of the  unit ball,  and we also give an  intrinsic characterization   by introducing the concept of Hausdorff dimension of  compact operators. The Dixmier trace of $K_\alpha$ are also calculated in this paper.

\end{abstract}
\maketitle

\section{Introduction}

This paper is a systematic research on the  $L^p$-$L^q$ compactness of Bergman-type integral operators on the unit ball. While there exist abundant works on the boundedness of Bergman-type integral operators, the investigation of compactness aspect was started only a few years.

Let us start with recalling  some notations and terminologies. Let $\mathbb{B}^d$ be the open  unit ball in the  usual complex Euclidian norm on the $d$-dimensional space $\mathbb{C}^d$ with the normalized Lebesgue measure $dv,$ which means the measure of  $\mathbb{B}^d$ is one. For any $\alpha\in\mathbb{R},$  the Bergman-type integral operator $K_\alpha$ on $L^1(\mathbb{B}^d,dv)$ is defined by
$$ K_\alpha f(z)=\int _{\mathbb{B}^d}\frac{f(w)}{(1-\langle z,w\rangle)^\alpha}dv(w),$$
where $\langle z,w \rangle=z_1\bar{w}_1+\cdots+z_d\bar{w}_d$ is the standard Hermitian inner product on  $\mathbb{C}^d.$ In particular, when $\alpha=d+1,$ $K_{d+1}$ is the standard Bergman projection on  $\mathbb{B}^d,$  since the function $K(z,w)=\frac{1}{(1-\langle z,w\rangle)^{d+1}}$ is the Bergman kernel of $\mathbb{B}^d$ with the  measure $dv.$ However, the Bergman-type integral operator $K_\alpha^+$  on $L^1(\mathbb{B}^d,dv)$   is given by 
$$ K_\alpha^+ f(z)=\int _{\mathbb{B}^d}\frac{f(w)}{|1-\langle z,w \rangle|^\alpha}dv(w).$$
Bergman-type integral operators play an important role in complex analysis and operator theory \cite{Pa,ZZ,Zhu,Zhu1,Zhu2}. In particular, for any $\alpha>0,$ if restrict $K_\alpha$ to Bergman spaces, then every $K_\alpha $ is a  special  fractional radial differential operator  \cite{ZZ,Zhu,Zhu1}   which is a kind of useful operators in the Bergman space theory on the unit ball, see Lemma \ref{kr} below. On the other hand, the operator $ K_\alpha^+$ in the complex analysis category, to some extent, is analogous to the Riesz potential  operators in the real analysis. The classical Riesz potential  operator   $R_\alpha$ is defined on the real Euclidian space $\mathbb{R}^d,$ whose  basic result
concerning mapping properties  is the Hardy-Littlewood-Sobolev theorem, which  characterizes   the boundedness of  Riesz potential  operator   $R_\alpha: L^p(\mathbb{R}^d)\rightarrow L^q(\mathbb{R}^d)$ and has been applied to the PDE theory for a long time; we refer the reader to \cite{Cui,Li,Pl}. 

 For convenience, we replace $L^p(\mathbb{B}^d,dv)$ by $L^p(\mathbb{B}^d)$ or  $L^p$  for any $1\leq p\leq\infty$ without confusion arises.  Bergman-type operators  $K_\alpha,K_\alpha^+$ are called  $L^p$-$L^q$ bounded (or compact) if  $K_\alpha,K_\alpha^+:L^p(\mathbb{B}^d)\rightarrow L^q(\mathbb{B}^d)$ are  bounded (or compact), where $1\leq p,q\leq\infty$.
Researches on  $L^p$-$L^q$ problems of Bergman operators $K_\alpha,K_\alpha^+$  go back to the  boundedness of Bergman projection on  bounded  domains. 
F. Forelli and W. Rudin \cite{FR}  proved  that  the Bergman projection  on the unit ball  is $L^p$-$L^p$ bounded if and only if $1<p<\infty,$  in their method an (Forelli-Rudin) asymptotic estimate of integral for Bergman kernel plays an important role; indeed, following the same method, one can   characterize the $L^p$-$L^p$ boundedness of more general Bergman-type operators, see \cite{Zhu,Zhu1}. Around the same time,  H. Phong and E. Stein \cite{PS} proved more  general results on a class of  bounded strongly pseudoconvex domains; see \cite{LS,Mc,MS} for more results along this line. 

 Recently, G. Cheng and X. Fang et al \cite{CF} completely characterized the $L^p$-$L^q$ boundedness of $K_\alpha$ on the unit disk $\mathbb{D},$ i.e. the case $d=1,$ and they  conjectured that there should exist  similar results in the high dimensional case. Their proofs depend on techniques of harmonic analysis and coefficient multiplier theory of Bergman space on the unit disk. Unfortunately, this method can not  be directly applied to the case of unit ball.  But, for the spacial case $\alpha=1,$  G. Cheng, C. Liu  et al \cite{CH} solved the  $L^p$-$L^q$  boundedness problem of $K_1,K_1^+$ on $\mathbb{B}^d.$  The author and his collaborator \cite{DiW} completely characterized  the  $L^p$-$L^q$ boundedness of $K_\alpha,K_\alpha^+$  and established some Hardy-Littlewood-Sobolev type inequalities on  $\mathbb{B}^d,$   motivated  by an observation in \cite{Zhu1} we further  investigated the $L^p$-$L^q$ compactness of $K_\alpha.$
 Techniques of  complex and  harmonic analysis are  synthetically utilized  in \cite{DiW}, such as the Carleson measure theory is applied to the  $L^p$-$L^q$ compactness problem. The key method used for compactness relies on the fact that  $K_\alpha$ maps $L^p$ space into Bergman space or Bloch space; however,  $K_\alpha^+$ has no such analytic property, thus  the  compactness characterization of  $K_\alpha^+$  remains open.

The first purpose of the present paper is to continue to  characterize the $L^p$-$L^q$ compactness of $K_\alpha^+.$  
The following two theorems are our main results on the  $L^p$-$L^q$ compactness.

{\noindent\bf Theorem 1\label{thm1}.} Suppose  $ d+1<\alpha<d+2,$ then the following statements are equivalent.
\begin{enumerate}
\item $K_\alpha:L^p\rightarrow L^q$ is bounded.
\item $K_{\alpha}^+:L^p\rightarrow L^q$ is bounded.
\item $K_\alpha:L^p\rightarrow L^q$ is compact.
\item $K_\alpha^+:L^p\rightarrow L^q$ is compact.
\item $p,q$ satisfy one of the following inequalities:\begin{enumerate}
\item[(a)]  $\frac{1}{d+2-\alpha}<p<\infty, \frac{1}{q}>\frac{1}{p}+\alpha-(d+1);$
\item [(b)] $p=\infty, q<\frac{1}{\alpha-(d+1)}.$
\end{enumerate}
\end{enumerate}
The equivalences of (1),(2),(3) and  (5) in Theorem 1 have been proved  in \cite{DiW}, from Theorem 1 we know the $L^p$-$L^q$ boundedness and compactness of operators $K_\alpha,K_\alpha^+$ are equivalent when  $ d+1<\alpha<d+2.$

{\noindent\bf Theorem 2\label{th2}.} Suppose  $ 0 <\alpha\leq d+1,$ then the following statements are equivalent:
\begin{enumerate}
\item $K_\alpha:L^p\rightarrow L^q$ is compact.
\item $K_\alpha^+:L^p\rightarrow L^q$ is compact.
\item $p,q$ satisfy one of the following inequalities:\begin{enumerate}
\item[(a)] $ p=1, q<\frac{d+1}{\alpha};$
\item[(b)] $1<p<\frac{d+1}{d+1-\alpha}, \frac{1}{q}> \frac{1}{p}+ \frac{\alpha}{d+1}-1;$
\item[(c)] $p=\frac{d+1}{d+1-\alpha}, q < \infty;$
\item[(d)] $\frac{d+1}{d+1-\alpha}<p\leq\infty.$
\end{enumerate}
\end{enumerate}
The equivalence of (1) and (3)  in Theorem 2 has been proved  in \cite{DiW}, combing Theorem 2 with   \cite[Theorem 2,3]{DiW}, we know the $L^p$-$L^q$ boundedness and compactness of operators $K_\alpha,K_\alpha^+$ are different  when  $ 0<\alpha\leq d+1.$  When $\alpha\leq0,$ Proposition \ref{fink} below shows  that  $K_\alpha,K_\alpha^+$ are all $L^p$-$L^q$  compact for any $1\leq p,q\leq\infty;$ moreover, $K_\alpha$  is a finite rank operator if $\alpha$ is a nonpositive integer and $K_\alpha^+$  is a finite rank operator if $\alpha$ is a nonpositive even integer. In contrast, when $\alpha\geq d+2,$  there  exist no $1\leq p,q\leq \infty $ such that $K_\alpha^+:L^p\rightarrow L^q$ is compact, see  Corollary \ref{kcor}  below.

 Theorem 2 and Proposition \ref{fink} imply   Bergman-type operator $K_\alpha $ are compact on the Hilbert space $L^2$ only when $ \alpha< d+1.$ As we all known, Schatten classes and Macaev classes are more refined classification of compact operators on Hilbert spaces, but also   involve global estimates of the spectrum of  compact operators. The next two theorems deal with the problem of  Schatten class and Macaev class Bergman-type operators.  Let $H$ be a separable Hilbert space,  denote by $\mathcal{L}^p(H)$ and $\mathcal{L}^{p,\infty}(H)$ the Schatten $p$-class (or ideal) and Macaev $p$-class (or ideal) on $H$ respectively, where $0 < p<\infty.$ More precisely, let $T$ be a compact operator on $H$ and \begin{equation}\label{pospe}\{\mu_n(\vert T\vert): \mu_0(\vert T\vert)\geq\mu_1(\vert T\vert)\geq\mu_2(\vert T\vert)\geq\cdots\}\end{equation}   be  the set of eigenvalues (counting multiplicities)  of  $\vert T\vert:=(T^\ast T)^{\frac{1}{2}},$ then $T\in \mathcal{L}^p(H)$ if and only if
 $\{\mu_n(\vert T\vert)\}\in \ell^p,$ and $T\in \mathcal{L}^{p,\infty}(H)$ if and only if  $\mu_n(\vert T\vert)=O(n^{-\frac{1}{p}}).$ We also use $\mathcal{L}^\infty(H)=\mathcal{L}^{\infty,\infty}(H)$ to stand for the compact operator ideal on $H.$
 Moreover, $\mathcal{L}^p(H)$ and $\mathcal{L}^{p,\infty}(H)$ will become  Banach spaces when provided suitable norms for $1 \leq p < \infty.$ We refer the reader to \cite{Con,Rin,Zhu1,Zhu2} for more details about Schatten class and Macaev class operators. Let us denote the set $\mathscr{S}_d\subset\mathbb{R}$ by $$\mathscr{S}_d=\{\alpha<d+1: -\alpha\notin \mathbb{N}\},$$ where $ \mathbb{N}$ is the set of all nonnegative integers.  For $1\leq p\leq \infty,$ let $A^p=\textrm{Hol}(\mathbb{B}^d)\cap L^p$ be the $p$-integrable  Bergman space on $\mathbb{B}^d.$ In particular, $A^2$ is a separable Hilbert space.

{\noindent\bf Theorem 3\label{th3}.} Suppose $ \alpha\in\mathscr{S}_d $ and $0 < p<\infty,$  then the following statements are equivalent.
\begin{enumerate}
\item $K_\alpha\in \mathcal{L}^p(L^2).$
\item $K_\alpha\in \mathcal{L}^p(A^2).$
\item $\widetilde{K_\alpha}\in L^p(d\lambda).$
\item $p>\frac{d}{d+1-\alpha}.$
\end{enumerate}
Where  $\widetilde{K_\alpha}$ is the Berezin transform of $K_\alpha$ on the Bergman space $A^2$  and $d\lambda$ is the M\"obius invariant measure on $\mathbb{B}^d,$ the exact definition will be given in  the following Section \ref{pfs}.

{\noindent\bf Theorem 4\label{th4}.} Suppose $ \alpha\in\mathscr{S}_d $ and $0 < p<\infty,$  then the following statements are equivalent.
\begin{enumerate}
\item $K_\alpha\in \mathcal{L}^{p,\infty}(L^2).$
\item $K_\alpha\in \mathcal{L}^{p,\infty}(A^2).$
\item $p\geq\frac{d}{d+1-\alpha}.$
\end{enumerate}
In particular, $K_\alpha\notin\mathcal{L}^{\frac{d}{d+1-\alpha}}$ but $K_\alpha\in\mathcal{L}^{\frac{d}{d+1-\alpha},\infty}$ when $ \alpha\in\mathscr{S}_d;$ however,  $K_\alpha\in\mathcal{L}^{\frac{d}{d+1-\alpha}}$ and $K_\alpha\in\mathcal{L}^{\frac{d}{d+1-\alpha},\infty}$ if $\alpha$ is a nonpositive integer. As an application of Theorem 3,4 we calculate the Dixmier trace of $K_\alpha$ whenever $K_\alpha\in\mathcal{L}^{1,\infty},$ see Proposition \ref{dixtr} below. Moreover, we remark that the  condition $ \alpha\in\mathscr{S}_d$ in Theorem 3,4 is necessary and  sharp, namely $\mathscr{S}_d$ is the maximal set ensures the theorems holds. Note that Theorem 3,4 are two results related to the underlying domain $\mathbb{B}^d.$ We also give a more intrinsic characterization  of Schatten class and Macave class Bergman-type operator $K_\alpha $  by introducing the concept of Hausdorff dimension of  compact operators, see Definition \ref{dimhau} and Corollary \ref{smcha} below. Hausdorff dimension of  compact operators is similar to the Hauadorff dimension of subsets in metric spaces \cite{LiY}.

It is worth mentioning that our results are totally new even in the one-dimensional  case and the results of this paper can be generalized to the weighted  Lebesgue integrable spaces and more general kernel operators on  the unit ball. Moreover, we in fact provide two  approaches to solve the Schatten class  Bergman-type operators on the unit ball, one of which essentially depend on the spectrum estimate without Forelli-Rudin estimate.  On the other hand, we know that the unit ball is a spacial bounded strongly pseudoconvex domain and every bounded strongly pseudoconvex domain with noncompact automorphism group is biholomorphic to the unit ball \cite{Won}, so we gauss that there exist some similar results about boundedness and compactness on  bounded strongly pesudoconvex domains.  

The paper is organized as follows. In Section 2, we give the proof of Theorem 1 by using a  criteria of precompactness in $ L^p$ space and the interpolation  of  compact operators.   Section 3 is devoted to complete the proof of Theorem 2, which is based on the hypergeometric function theory and  the  fractional radial differential operator theory. We also describe the phenomenon of Bergman-type operators for $\alpha\leq0.$ In Section 4, we will characterize Schatten class and Macaev class Bergman-type operator $K_\alpha$ by   estimates of the spectrum and  techniques of operator theory.

\section{The case of $d+1<\alpha<d+2$}
In this section, we will prove the main Theorem 1.  We need some lemmas. The following lemma gives the regularity of the image of $K_\alpha,K_\alpha^+.$
\begin{lem}\label{klp} For any $\alpha\in \mathbb{R}$ and $f\in L^p(\mathbb{B}^d),1\leq p\leq \infty,$ then the followings hold.
\begin{enumerate}
\item $K_\alpha f$ is holomorphic on $\mathbb{B}^d.$
\item $K_\alpha^+f$ is smooth on $\mathbb{B}^d.$
\end{enumerate}
\end{lem}
\begin{proof} (1) It is sufficient to show that $K_\alpha f,f\in L^1(\mathbb{B}^d)$ is holomorphic  on every point of  $\mathbb{B}^d.$  Suppose $z$ is an arbitrary point in $ \mathbb{B}^d.$ Denote $B_{r}=\{w\in \mathbb{C}^d: \vert w\vert <r\},r>0.$ Choose $\vert z\vert<r_0<r_0'<1,$ and then choose $1<r_1<\frac{1}{r_0'}.$ It is easy to see that $\{z\}\times \mathbb{B}^d\subset B_{r_0}\times   \mathbb{B}^d \subset B_{r_0'}\times  B_{r_1}$ and $0<r_0<r_0r_1<1.$ It follows that the function $K_\alpha(z,w)=\frac{1}{(1-\langle z,\bar{w}\rangle)^\alpha} $ is holomorphic on the domain $B_{r_0'}\times  B_{r_1}.$ Then
 \cite[ Proposition 1.2.6]{Rud} implies that  the binary function function $K_\alpha(z,w)=\frac{1}{(1-\langle z,\bar{w}\rangle)^\alpha} $ has a global power series expansion on $B_{r_0'}\times  B_{r_1}.$ Therefor, we can suppose that
\begin{equation}\label{ser} K_\alpha(z,w) =\sum_{n,m\geq0}a_{n,m}z^n\bar{w}^m,\end{equation}  on $B_{r_0'}\times  B_{r_1}.$ Here $n,m$ are multi-indexes,  $n=(n_1,\cdots,n_d)\geq0$ means $n_j\geq0$ for any $j$, and  $z^n=\prod_{j=1}^d  z_j^{n_j}.$
Note that $ \overline{B_{r_0}}\times  \overline{ \mathbb{B}^d}\subset B_{r_0'}\times  B_{r_1},$ we get that the power series (\ref{ser}) is uniformly converge on  compact set $\overline{B_{r_0}}\times  \overline{ \mathbb{B}^d}.$ Then the dominated convergence theorem and the maximum modulus principle of analytic function imply that $$K_\alpha f(z)=\sum_{n,m\geq0}a_{n,m}\left(\int_{\mathbb{B}^d}f\bar{w}^mdv\right)z^n,$$ on $B_{r_0}\ni z.$ This implies that $K_\alpha f$ is holomorphic at   $z\in\mathbb{B}^d.$

(2) Similarly,  it suffices  to consider the case $f\in L^1.$  Denote the holomorphic function $G_\alpha(z,w,u,v)$ on $B_{r_0'}\times  B_{r_1}\times  B_{r_1}\times B_{r_0'}$ by $$G_\alpha(z,w,u,v)=\frac{1}{(1-\langle z,\bar{w}\rangle)^\frac{\alpha}{2}}\frac{1}{(1-\langle u,\bar{v}\rangle)^\frac{\alpha}{2}}.$$
By  \cite[Proposition 1.2.6]{Rud} again, we know that $G_\alpha(z,w,u,v)$ has a global power series expansion, thus we can suppose that
$$G_\alpha(z,w,u,v)=\sum_{n,m,k,l\geq0}a_{n,m,k,l}z^nw^mu^kv^l$$ on $B_{r_0'}\times  B_{r_1}\times  B_{r_1}\times B_{r_0'}.$ Since  $$ \overline{B_{r_0}}\times \overline{ \mathbb{B}^d} \times \overline{ \mathbb{B}^d}  \times \overline{B_{r_0}}\subset  B_{r_0'}\times  B_{r_1}\times  B_{r_1}\times B_{r_0'},$$ it follows from  the dominated convergence theorem and the maximum modulus principle of analytic function   that

 \begin{equation}\begin{split}\label{}
F_\alpha(u,v)&\overset{\text{def}}{=}\int_{\mathbb{B}^d}f(w)G_\alpha(u,\bar{w},w,v)dv(w)\notag\\
&=\sum_{n,m,k,l\geq0}a_{n,m,k,l}\left(\int_{\mathbb{B}^d}fw^k\bar{w}^mdv\right)u^nv^l,
\end{split}\end{equation}
is holomorphic on ${B_{r_0}}\times{B_{r_0}}.$ Then
 \begin{equation}\begin{split}\label{kfz}
 K_\alpha^+f(z)=\int _{\mathbb{B}^d}\frac{f(w)dv(w)}{|1-\langle z,w \rangle|^\alpha}= \int _{\mathbb{B}^d}f(w)G_\alpha(z,\bar{w},w,\bar{z})dv(w)=F_\alpha(z,\bar{z})\\
\end{split}\end{equation}
 on $B_{r_0}\ni z.$ It shows that  $K_\alpha^+f$ is smooth at   $z\in\mathbb{B}^d.$\end{proof}
\begin{rem} The formula (\ref{kfz}) in fact shows that, for any $f\in L^1,$ the function $K_\alpha^+f$ admits a global power series expansion in the form
$$ K_\alpha^+f(z)=\sum_{n,m\geq0}a_{n,m}z^n\bar{z}^m,$$ and $K_\alpha^+f$ uniquely determines a holomorphic function on $\mathbb{B}^d\times \mathbb{B}^d.$
\end{rem}

 From Lemma \ref{klp}, we know that the image of $K_\alpha^+$ is continuous. The next lemma provides a result on  the equicontinuity.
 \begin{lem}\label{uf} Suppose $\alpha\in\mathbb{R}, 1\leq p\leq \infty$  and $K\subset \mathbb{B}^d$ is  a compact subset, then for any $\varepsilon>0,$ there exists a  $\delta>0$ such that
\begin{equation}\label{ku} \left\vert K_\alpha^+f(z_1)- K_\alpha^+f(z_2)\right\vert\leq \Vert f\Vert_p\cdot\varepsilon, \quad\forall f\in L^p( \mathbb{B}^d),\notag\\
\end{equation}
whenever $z_1,z_2\in K$ and $\vert z_1-z_2\vert<\delta.$
\end{lem}
\begin{proof} Let $B_r$ be the ball with radius $r$ defined in Lemma \ref{klp}. Since  $K\subset \mathbb{B}^d$ is  a compact subset, there exists a $0<r<1$ such that
$K\subset B_r\subset\overline{B_r}\subset\mathbb{B}^d.$ Since the function $\frac{1}{\vert 1-\langle z,w \rangle\vert^\alpha}$ is uniformly continuous on the compact set $\overline{B_r}\times\overline{\mathbb{B}^d},$ it yields that,  for any $\varepsilon>0,$ there exists a  $\delta>0,$ such that
$$\left\vert\frac{1}{\vert 1-\langle z_1,w \rangle\vert^\alpha}-\frac{1}{\vert 1-\langle z_2,w \rangle\vert^\alpha}\right\vert\leq\varepsilon,$$
whenever $z_1,z_2\in K$ and $\vert z_1-z_2\vert<\delta.$ Thus, for any $ f\in L^p( \mathbb{B}^d),$ we have
$$\left\vert K_\alpha^+f(z_1)- K_\alpha^+f(z_2)\right\vert\leq\int _{\mathbb{B}^d}\vert f\vert\left\vert\frac{1}{|1-\langle z_1,w \rangle|^\alpha}-\frac{1}{|1-\langle z_2,w \rangle|^\alpha}\right\vert dv(w)\leq \Vert f\Vert_p\cdot\varepsilon$$
whenever $z_1,z_2\in K$ and $\vert z_1-z_2\vert<\delta.$ It completes the proof.
\end{proof}

A subset in a Banach space is called precompact if the closure of the subset in the norm topology is compact. Obviously, an operator between two Banach spaces is compact if and only if the operator maps every bounded set to precompact one.  Suppose $\Omega\subset \mathbb{C}^d$ is an arbitrary  bounded domain. The following lemma  provides a criteria of precompactness in $ L^p(\Omega)$ with $1\leq p<\infty,$ see \cite[Theorem 2.33]{AdF} for more details.

\begin{lem}\cite{AdF} \label{AdF} Let $1\leq p<\infty$ and $K\subset L^p(\Omega).$ Suppose there exists a sequence $\{\Omega_j\}$ of subdomains of $\Omega$ having the following properties:
\begin{enumerate}
\item  $\Omega_j\subset\Omega_{j+1}.$
\item The set of restrictions to $\Omega_j$ of the functions in $K$ is precompact in  $L^p(\Omega_j)$ for each $j.$
\item For every $\varepsilon>0$ there exists a $j$ such that $$ \int_{\Omega-\Omega_j}\vert f\vert^pdv<\varepsilon,\quad\forall f \in K.$$
\end{enumerate}
Then $K$ is precompact in $ L^p(\Omega).$
\end{lem}

The  Forelli-Rudin asymptotic estimate of integral for Bergman kernel on the unit ball  $\mathbb{B}^d$ is a fundamental result in function spaces and operator theory, see \cite[Propostion 1.4.10]{Rud} or \cite[Theorem 1.12]{Zhu1} for more details.  In the rest of this article, we will frequently use this estimate, for the sake of convenience, we state as follows.
\begin{lem} (Forelli-Rudin)\label{FR} Suppose $c$ is real and $t>-1$. Then the integral $$J_{c,t}(z)=\int_{\mathbb{B}^d}\frac{(1-\vert w\vert^2)^tdv(w)}{\vert 1-\langle z,w\rangle\vert ^{c+t}}, \quad z\in\mathbb{B}^d,$$
has the following asymptotic properties.
\begin{enumerate}
\item If $c < d+1$, then $J_{c,t}$ is  bounded on $\mathbb{B}^d.$
\item If $c = d+1,$ then
$$J_{c,t}(z)\sim -\log (1-\vert z\vert^2),\quad \vert z\vert \rightarrow{1}^-.$$
\item  If $c >d+1,$ then $$J_{c,t}(z)\sim (1-\vert z\vert^2)^{d+1-c},\quad \vert z\vert \rightarrow{1}^-.$$
\end{enumerate}
The notation $A(z)\sim B(z)$ means that the ratio $\frac{A(z)}{B(z)}$ has a positive finite limit as $\vert z\vert \rightarrow 1^-.$
\end{lem}

 To complete the proof of  Theorem 1, we also need  a result on the interpolation  of  compact operators. The following lemma is first proved by A. Krasnoselskii \cite{Kr}, see also \cite{CP} or \cite[ Theorem 3.10]{Kr1}.
 \begin{lem} \cite{CP,Kr,Kr1}\label{cpt}  Suppose that $1\leq p_1,p_2,q_1,q_2\leq\infty$ and $q_2\neq\infty.$  If a linear operator  $T$ such that $T :L^{p_1}\rightarrow L^{q_1}$ is bounded and $T :L^{p_2}\rightarrow L^{q_2}$ is compact, then $T :L^{p}\rightarrow L^{q}$ is compact, if   there exists a  $\theta\in (0,1)$ such that  $$\frac{1}{p}=\frac{\theta}{p_1}+\frac{1-\theta}{p_2}, \frac{1}{q}=\frac{\theta}{q_1}+\frac{1-\theta}{q_2}.$$
\end{lem}

 {\noindent{\bf{Proof of  Theorem 1.}} The equivalence of $(1)\Leftrightarrow (2) \Leftrightarrow(3)\Leftrightarrow (5)$ is the main result of  \cite[Theorem 1]{DiW}, so it is enough to prove $(4)\Leftrightarrow (5).$ Note that the equivalence of (3) and (5), it implies that $(4)\Rightarrow (5),$ since compact operators are all bounded between two Banach spaces. Now we prove $(5)\Rightarrow (4),$ which means we need to show $K_\alpha^+$ is $L^p$-$L^q$ compact if  $K_\alpha^+$ is $L^p$-$L^q$ bounded. It is easy to see  that $1\leq q<p\leq \infty$ under  the  assumption in (5). By Lemma \ref{cpt}, the proof will be completed once
we  prove the following two conclusions:
 \begin{enumerate}
\item[(a)]  If $p=\infty,q<\frac{1}{\alpha-(d+1)},$ then $K_\alpha^+: L^p \rightarrow L^q$ is compact.
\item[(b)]  If $p>\frac{1}{(d+2)-\alpha },q=1,$  then $K_\alpha^+: L^p \rightarrow L^q$ is compact.
\end{enumerate}
Observe the operator $K_\alpha^+$ is adjoint by Fubini's theorem, combing this with the well known fact that an operator between two Banach spaces is compact if and only if its  adjoint is compact, we conclude that  conclusions (a) and (b) are in fact  equivalent. Consequently, the proof is completed if conclusion (a) is proved. Now we turn to prove the conclusion (a). Suppose $\{f_j\}$ is  an arbitrary bounded sequence in $L^\infty,$ without loss of generality, we can suppose that
   \begin{equation}\label{bds} \Vert f_j\Vert_\infty \leq 1,\quad j=1,2,\cdots.\end{equation}
   Denote the  bounded domain $B_j'$ by $B_j'=\{z\in\mathbb{C}^d:\vert z\vert<1-\frac{1}{j}\},j=1,2,\cdots.$ Clearly,   $B_j'$ is compactly contained in $\mathbb{B}^d$ and $B_j'\subset B_{j+1}'\subset \mathbb{B}^d,$ for every $j.$ We first prove that the  set of restrictions to $B_j'$ of the functions in $\{K_\alpha^+f_n\}$ is precompact in  $L^p(B_j')$ for each $j.$ In view of Lemma \ref{klp}, we know the functions in $\{K_\alpha^+f_n\}$ are all continuous on  $\mathbb{B}^d$ and uniformly  continuous  on every $\overline{B_j'}.$ Combing with the fact that the embedding $C(\overline{B_j'})\subset L^p(B_j')$ is continuous for every $j,$ it is enough to prove $\{(K_\alpha^+f_n)|_{B_j'}\}$ is precompact in $C(\overline{B_j'})$ for every $j.$
Note that (\ref{bds}) and the definition of $B_j',$ we have that
\begin{equation}\begin{split}\label{Kaf}
\Vert ( K_\alpha^+ f_n)|_{B_j'}\Vert_\infty=\sup_{z\in \overline{B_j'}}\vert K_\alpha^+ f_n\vert &=\sup_{z\in \overline{B_j'}}\left\vert\int_{\mathbb{B}^d}\frac{f_n(w)dv(w)}{\vert 1-\langle z,w\rangle\vert^\alpha}\right\vert\\\
 &\leq\Vert f_n\Vert_\infty\sup_{z\in \overline{B_j'}}\int_{\mathbb{B}^d}\frac{dv(w)}{\vert 1-\langle z,w\rangle\vert^\alpha}\\
&\leq j^\alpha\Vert f_n\Vert_\infty\\
&\leq j^\alpha,
\end{split}\end{equation} for every $j.$ The estimate (\ref{Kaf}) implies that  $\{(K_\alpha^+f_n)|_{B_j'}\}$ are bounded in $C(\overline{B_j'})$ for every $j.$ It follows from Lemma \ref{uf} that $\{(K_\alpha^+f_n)|_{B_j'}\}$ is equicontinuous on $\overline{B_j'}$ for every $j.$ Then  Arzel\`a-Ascoli theorem implies that $\{(K_\alpha^+f_n)|_{B_j'}\}$ is  precompact in $C(\overline{B_j'})$ for every $j.$ By the well known fact that, for $t\in \mathbb{R},$ $(1-\vert z\vert^2)^t \in L^1(\mathbb{B}^d)$ if and only if $t>-1,$
 we obtain  that, for any fixed $t>-1$ and for any $\varepsilon>0,$ there exists a $J>0$ satisfying \begin{equation}\label{ie} \int_{\mathbb{B}^d-B_j'}(1-\vert z\vert^2)^t dv<\varepsilon,\quad \forall j>J,\end{equation}  since  the absolute continuity of the integral and $\lim_{j\rightarrow\infty}v(\mathbb{B}^d-B_j')=0.$  From the assumption in (a) and $d+1<\alpha<d+2,$ we get that $0<q(\alpha-(d+1))<1,$ then by (\ref{ie}) we obtain that for any $\varepsilon>0,$ there exists a $J>0$ such that \begin{equation}\label{iee} \int_{\mathbb{B}^d-B_j'}(1-\vert z\vert^2)^{-q(\alpha-(d+1))} dv<\varepsilon,\quad \forall j>J.\end{equation}
Combing (\ref{iee}) with Lemma \ref{FR},  there exists a positive constant $C $  such that, for any $\varepsilon>0,$ there exists a $J>0$ satisfying
\begin{equation}\begin{split}
\int_{\mathbb{B}^d-B_j'}\vert K_\alpha^+f_n(z)\vert^qdv(z)&=\int_{\mathbb{B}^d-B_j'}\left\vert\int_{\mathbb{B}^d}\frac{f_n(w)dv(w)}{\vert 1-\langle z,w\rangle\vert^\alpha}\right\vert^qdv(z)\notag\\
&\leq \Vert f_n\Vert_\infty^q\int_{\mathbb{B}^d-B_j'}\left\vert\int_{\mathbb{B}^d}\frac{dv(w)}{\vert 1-\langle z,w\rangle\vert^\alpha}\right\vert^qdv(z)\\
&\leq C \Vert f_n\Vert_\infty^q \int_{\mathbb{B}^d-B_j'} (1-\vert z\vert^2)^{-q(\alpha-(d+1))} dv\\
&\leq C\varepsilon,\\
\end{split}\end{equation}
for any $j>J$ and $n=1,2,\cdots.$ Thus $\{K_\alpha^+f_n\}$ is precompact in $L^q$ by Lemma \ref{AdF}, it completes the proof.
\qed
 \begin{rem} Theorem 1 indicates  that Bergman operators $K_\alpha,K_\alpha^+ $ from $L^p$ into $L^q$ are compact if and only if they are bounded when $1\leq q<p\leq\infty.$ 
 In fact, this is quite expected, which can be seen from  Ando's theorem for a bounded integral  operator from $L^p$ into $L^q$ with $1\leq q<p\leq\infty,$  see \cite{An} or  \cite[Theorem 5.5, 5.13]{Kr1}.    Nevertheless, we give a direct proof here.

 \end{rem}

\section{The case of $\alpha\leq d+1$}

This section is  mainly  devoted to prove Theorem 2, which characterizes the $L^p$-$L^q$ compactness of $K_\alpha^+ $ for $0<\alpha\leq d+1.$  We  first recall some results on hypergeometric function theory. For complex numbers $\alpha,\beta,\gamma$ and complex variable $z,$ we use the classical notation $\tensor[_2]{F} {_1}(\alpha,\beta;\gamma;z) $  to denote
$$\tensor[_2]{F} {_1}(\alpha,\beta;\gamma;z) =\sum_{j=0}^{\infty}\frac{(\alpha)_j(\beta)_j}{j!(\gamma)_j}z^j,$$ with $\gamma\neq 0,-1,-2,\ldots,$  where $(\alpha)_j=\Pi_{k=0}^{j-1}(\alpha+k)$ is the Pochhammer for any complex number $\alpha.$ The following lemma calculates the
exact value of  the hypergeometric function at the point $z=1.$
\begin{lem}\cite[Section 2.8]{EMO} \label{hi2} If $Re(\gamma-\alpha-\beta)>0,$ then
$$\tensor[_2]{F} {_1}(\alpha,\beta;\gamma;1)=\frac{\Gamma(\gamma)\Gamma(\gamma-\alpha-\beta)}{\Gamma(\gamma-\alpha)\Gamma(\gamma-\beta)},$$ 
where $\Gamma$ is the usual Euler Gamma  function.
\end{lem}

Recall the definition of  the function $J_{\beta,\gamma}$ in Lemma \ref{FR}.   The following lemma is  not only  a restatement of Lemma \ref{FR}, but also establishes a  connection between the integration induced by Bergman kernel  and the  hypergeometric function. A function is called finite if it has finite value at every point of its domain.
\begin{lem}\label{hi1} (1)\cite{Rud} Suppose $\beta\in\mathbb{R}$ and $\gamma>-1,$ then
 \begin{equation}\label{jg} J_{\beta-\gamma,\gamma}(z)=\frac{\Gamma(1+d)\Gamma(1+\gamma)}{\Gamma(1+d+\gamma)}\tensor[_2]{F} {_1}(\frac{\beta}{2},\frac{\beta}{2};1+d+\gamma;|z|^2), \end{equation} for $z\in\mathbb{B}^d.$

(2)  $J_{\beta-\gamma,\gamma}$ is finite  on the closed ball $\overline{\mathbb{B}^d}$ if and only  if  $\beta<d+1+\gamma.$  Moreover, in this case, the identity (\ref{jg})  actually holds on the closed ball $\overline{\mathbb{B}^d}.$
\end{lem}
 \begin{proof} (1) In this case, the identity (\ref{jg}) is a restatement of Lemma \ref{FR}, we refer the reader to \cite[Proposition 1.4.10]{Rud} for more details.

 (2) It suffices to prove $J_{\beta-\gamma,\gamma}$ is finite on the unit sphere $\mathbb{S}^d=\{ z\in \mathbb{C}^d:\vert z\vert =1\}$ if and only if  $\beta<d+1+\gamma.$ Due to  the unitary invariance of  the Lebesgue measure, we know that $$J_{\beta-\gamma,\gamma}(\eta)=\int_{\mathbb{B}^d}\frac{(1-\vert w\vert^2)^\gamma dv(w)}{\vert 1-\langle \eta,w\rangle\vert ^{\beta}}=\int_{\mathbb{B}^d}\frac{(1-\vert w\vert^2)^\gamma dv(w)}{\vert 1-w_1\vert ^{\beta}},$$  for any $ \eta\in\mathbb{S}^d.$ Then \cite[Lemma 1.8,1.9]{Zhu} imply that \begin{equation}\begin{split}\label{jbr}  J_{\beta-\gamma,\gamma}(\eta)
 &=2d\int_0^1r^{2d-1}(1-r^2)^\gamma dr\int_{\mathbb{S}^d} \frac{1}{\vert 1-r\xi_1\vert^\beta}d\sigma(\xi) \\
 &=2d\int_0^1r^{2d-1}(1-r^2)^\gamma  (d-1)\int_{\mathbb{B}^1} \frac{(1-\vert w\vert^2)^{d-2}}{\vert 1-rw\vert^\beta}dv_1(w)dr \\
 &= 2d\int_0^1r^{2d-1}(1-r^2)^\gamma \tensor[_2]{F} {_1}(\frac{\beta}{2},\frac{\beta}{2};d;r^2)dr\\
 &=2d\sum_{j=0}^{\infty}\frac{(\frac{\beta}{2})_j(\frac{\beta}{2})_j}{j!(d)_j} \int_0^1r^{2j+2d-1}(1-r^2)^\gamma dr\\
 &=\frac{\Gamma(d+1)\Gamma(\gamma+1)}{\Gamma(d+\gamma+1)}\sum_{j=0}^{\infty}\frac{(\frac{\beta}{2})_j(\frac{\beta}{2})_j}{j!(d+\gamma +1)_j}\\
 \end{split}\end{equation}
  for any $ \eta\in\mathbb{S}^d,$ where $d\sigma$ and $dv_1$ are  normalized Lebesgue measures on $\mathbb{S}^d$ and  $\mathbb{B}^1,$  respectively.  Note that $\Gamma(s+j)= (s)_j\Gamma(s)$  for all $s\in\mathbb{C}$ except  the nonpositive integers. Then Stirling's formula and (\ref{jbr}) yield that  $J_{\beta-\gamma,\gamma}$ is finite on $\mathbb{S}^d$ if and only if  $\beta<d+1+\gamma.$ From Lemma \ref{hi2} and (\ref{jbr}), we kwon that, if $\beta<d+1+\gamma,$ then
   \begin{equation}\begin{split}\label{jbrr}  J_{\beta-\gamma,\gamma}(\eta) &=\frac{\Gamma(d+1)\Gamma(\gamma+1)}{\Gamma(d+\gamma+1)}\sum_{j=0}^{\infty}\frac{(\frac{\beta}{2})_j(\frac{\beta}{2})_j}{j!(d+\gamma +1)_j}\\
     &= \frac{\Gamma(d+1)\Gamma(\gamma+1)}{\Gamma(d+\gamma+1)} \tensor[_2]{F} {_1}(\frac{\beta}{2},\frac{\beta}{2};d+\gamma+1;1)\\
  &=\frac{\Gamma(d+1)\Gamma(\gamma+1)\Gamma(d+\gamma+1-\beta)}{\Gamma^2(d+\gamma+1-\frac{\beta}{2})}\\
   \end{split}\end{equation}
    for any $ \eta\in\mathbb{S}^d,$ which means that the identity (\ref{jg}) also holds on the closed ball $\mathbb{S}^d.$ It completes the proof.
  \end{proof}
From Lemma \ref{hi1}, we know that $J_{\beta-\gamma,\gamma}$  is a  radial  function on $\mathbb{B}^d.$ Moreover, when $\beta<d+1+\gamma,$ $J_{\beta-\gamma,\gamma}$ is  increasing on $ \overline{\mathbb{B}^d}$ in the following sense,  \begin{equation}\label{jin} J_{\beta-\gamma,\gamma}(z_1)< J_{\beta-\gamma,\gamma}(z_2), \end{equation} whenever  $\vert z_1\vert <\vert z_2\vert\leq 1,$ since all Taylor coefficients of the hypergeometric function in (\ref{jg}) are positive.
Now, we introduce the following auxiliary function $I_\alpha$ for every $\alpha<d+1.$ The function $I_\alpha(r,z)$ on $[0,1)\times \overline{\mathbb{B}^d} $ is denoted  by
$$I_\alpha(r,z)=\int_{r\leq \vert w\vert<1}\frac{1}{\vert 1-\langle z,w\rangle\vert^{\alpha}}dv(w).$$
Since $\alpha<d+1,$  it follows by (\ref{jbrr}) and  (\ref{jin})  that
 \begin{equation}\begin{split}\label{ijga} I_\alpha(r,z)&\leq J_{\alpha,0}(z)\\
 &\leq\sum_{j=0}^\infty\left( \frac{\Gamma(j+\frac{\alpha}{2})}{\Gamma(\frac{\alpha}{2})}\right)^2\frac{\Gamma(d+1)}{\Gamma(j+1)\Gamma(j+d+1)}\\
 &= \frac{\Gamma(d+1)\Gamma(d+1-\alpha)}{\Gamma^2(d+1-\frac{\alpha}{2})},
   \end{split}\end{equation}
 for any $(r,z)\in [0,1)\times \overline{\mathbb{B}^d},$ which means that $I_\alpha$ is finite on $[0,1)\times \overline{\mathbb{B}^d}.$ Moreover, $I_\alpha$ is   increasing on $ [0,1)\times \overline{\mathbb{B}^d}$ in the following sense.
\begin{lem}\label{Lia} Suppose $r\in [0,1),$  then $$I_\alpha(r,z_1)\leq I_\alpha(r,z_2),$$ whenever $z_1,z_2\in \overline{\mathbb{B}^d}$ and $\vert z_1\vert \leq \vert z_2\vert.$
\end{lem}
\begin{proof} From (\ref{ijga}), we know that  $I_\alpha$ is finite on $[0,1)\times \overline{\mathbb{B}^d}.$ Now we calculate its exact value. It follows  from   \cite[Lemma 1.8,1.11]{Zhu} and the unitary invariance of  the Lebesgue measure  that
 \begin{equation}\begin{split}\label{alz}  I_\alpha(r,z)
&= \int_{r\leq \vert w\vert<1}\frac{1}{\vert 1-\vert z\vert w_1\vert^{\alpha}}dv(w)\\
 &=\int_{r\leq \vert w\vert<1}\sum_{j=0}^\infty\left( \frac{\Gamma(j+\frac{\alpha}{2})}{\Gamma(\frac{\alpha}{2})\Gamma(j+1)}\right)^2\vert z\vert^{2j}\vert w_1^j\vert^{2} dv(w)\\
 &=\sum_{j=0}^\infty\left( \frac{\Gamma(j+\frac{\alpha}{2})}{\Gamma(\frac{\alpha}{2})\Gamma(j+1)}\right)^2\vert z\vert^{2j}\cdot2d\int_{r}^1t^{2d+2j-1}dt\int_{\mathbb{S}^d}\vert \xi_1^j\vert^{2} d\sigma(\xi)\\
&=\sum_{j=0}^\infty\left( \frac{\Gamma(j+\frac{\alpha}{2})}{\Gamma(\frac{\alpha}{2})}\right)^2\frac{\Gamma(d+1)(1-r^{2(j+d)})}{\Gamma(j+1)\Gamma(j+d+1)}\vert z\vert^{2j} \\\end{split}\end{equation}
for any $z\in\overline{\mathbb{B}^d}.$
It leads to the desired result since all coefficients of the power series expansion about $\vert z\vert$ in (\ref{alz}) are positive.
\end{proof}
\begin{prop} \label{kcom} If $0<\alpha<d+1,$  then $K_\alpha^+:L^\infty \rightarrow L^q$ is compact for any  $ 1\leq q\leq\infty.$
\end{prop}
\begin{proof} It suffices to prove that $K_\alpha:L^\infty \rightarrow L^\infty$ is compact. It is clear that $K_\alpha^+ f$ is continuous on the open ball $\mathbb{B}^d$ for any $f\in L^\infty$ by Lemma \ref{klp}. From Lemma \ref{hi1}, (\ref{jbrr}) and (\ref{jin}), we obtain that $K_\alpha^+ f(\eta)$ exists  and $$\vert K_\alpha^+ f(\eta)\vert \leq\Vert f\Vert_\infty \frac{\Gamma(d+1)\Gamma(d+1-\alpha)}{\Gamma^2(d+1-\frac{\alpha}{2})}$$ for any $\eta\in \mathbb{S}^d.$  We now turn to prove that $K_\alpha^+ f$ is continuous on $\overline{\mathbb{B}^d}.$ It suffices to prove that  $K_\alpha^+f$  on $\mathbb{S}^d.$ Thus we need only to prove that, for any $\eta\in \mathbb{S}^d$ and for any point sequence $\{ z_n\}$ in $\mathbb{B}^d $ satisfying   $z_n\rightarrow \eta,$ then  $K_\alpha^+ f(z_n)\rightarrow K_\alpha^+ f(\eta)$ as $n\rightarrow \infty.$ By Lemma \ref{hi1} and  Lemma \ref{jin}, we have

 \begin{equation}\begin{split}\label{kiie}\vert K_\alpha^+ f(z)\vert \leq \Vert f\Vert_\infty J_{\alpha,0}(z)
 \leq \Vert f\Vert_\infty J_{\alpha,0}(\eta)
 = \Vert f\Vert_\infty \frac{\Gamma(d+1)\Gamma(d+1-\alpha)}{\Gamma^2(d+1-\frac{\alpha}{2})},
 \end{split}  \end{equation}
 for any $z\in \mathbb{B}^d.$ The absolute continuity of the integral implies that, for any $\varepsilon>0,$ there exists a $0<\delta<1$, satisfying  \begin{equation}\label{iif}\int_{F}\frac{dv(w)}{\vert 1-\langle \eta,w\rangle\vert^\alpha}\leq \frac{\varepsilon}{4},\end{equation}  whenever $v(F)<\delta.$ Denote $F_\delta=\{z\in \mathbb{B}^d: \sqrt[d]{1-\frac{\delta}{2}}<\vert z\vert<1\}.$ Note that $v(F_\delta)=\frac{\delta}{2}<\delta$  and $$\frac{1}{\vert 1-\langle z_n,w\rangle\vert^\alpha}\rightarrow \frac{1}{\vert 1-\langle \eta ,w\rangle\vert^\alpha}~\text{uniformly~on }~\mathbb{B}^d \setminus F_\delta,\text{~as}~n\rightarrow \infty.$$
Then there exists a $N>0$ such that, for any $n>N,$ $$\int_{\mathbb{B}^d\setminus F_\delta}\left\vert\frac{1}{\vert 1-\langle z_n,w\rangle\vert^\alpha}-\frac{1}{\vert 1-\langle \eta,w\rangle\vert^\alpha}\right\vert dv(w)\leq \frac{\varepsilon}{2}.$$
Combing this with Lemma \ref{Lia} and (\ref{iif}), we conclude that, for any $n>N,$
 \begin{equation}\begin{split}\label{keps} \vert K_\alpha^+ f(z_n)-K_\alpha^+ f(\eta)\vert&\leq \Vert f\Vert_\infty\int_{\mathbb{B}^d\setminus F_\delta}\left\vert\frac{1}{\vert 1-\langle z_n,w\rangle\vert^\alpha}-\frac{1}{\vert 1-\langle \eta,w\rangle\vert^\alpha}\right\vert dv(w) \\
&~~~~~~~~~+\Vert f\Vert_\infty\int_{F_\delta}\left\vert\frac{1}{\vert 1-\langle z_n,w\rangle\vert^\alpha}-\frac{1}{\vert 1-\langle \eta,w\rangle\vert ^\alpha}\right\vert dv(w)\\
&\leq \Vert f\Vert_\infty\int_{\mathbb{B}^d\setminus F_\delta}\left\vert\frac{1}{\vert 1-\langle z_n,w\rangle\vert ^\alpha}-\frac{1}{\vert 1-\langle \eta,w\rangle\vert ^\alpha}\right\vert dv(w) \\
&~~~~~~~~~+2\Vert f\Vert_\infty\int_{F_\delta}\frac{1}{\vert 1-\langle \eta,w\rangle\vert^\alpha} dv(w)\\
&\leq  \Vert f\Vert_\infty\frac{\varepsilon}{2}+2 \Vert f\Vert_\infty\frac{\varepsilon}{4}\\
&= \varepsilon \Vert f\Vert_\infty.\\
\end{split}  \end{equation}
It completes the proof that $K_\alpha^+f$ is continuous on  the closed ball $\overline{\mathbb{B}^d}$ for any $f\in L^\infty.$  Now we prove that, for any bounded sequence  in $L^\infty,$ there exists a subsequence satisfying that its image under $K_\alpha^+ $ is convergent in $L^\infty.$ Suppose that $\{f_n\}$ is a bounded sequence in $L^\infty,$ then we have that $\{K_\alpha^+ f_n\}$ is in $C(\overline{\mathbb{B}^d})$ and $\{K_\alpha^+ f_n\}$ is uniformly bounded by (\ref{kiie}).  Now we prove that $\{K_\alpha^+ f_n\}$ is also equicontinuous.
 From (\ref{keps}), we know that \begin{equation}\label{limv}\lim_{\mathbb{B}^d\ni z\rightarrow\eta} \int_{\mathbb{B}^d}\left\vert\frac{1}{\vert 1-\langle z,w\rangle\vert^\alpha}-\frac{1}{\vert 1-\langle \eta,w\rangle\vert ^\alpha}\right\vert dv(w)=0,\end{equation} for arbitrary fixed $\eta\in\mathbb{S}^d.$ Combing (\ref{limv}) with the unitary invariance of Lebesgue measure and the symmetry of the unit ball, we have that,  for any $\varepsilon>0,$ there exists a $0<\delta'<1,$ satisfying that  \begin{equation}\label{limee}\int_{\mathbb{B}^d}\left\vert\frac{1}{\vert 1-\langle z,w\rangle\vert^\alpha}-\frac{1}{\vert 1-\langle \eta,w\rangle\vert^\alpha}\right\vert dv(w)\leq\frac{\varepsilon}{2}\end{equation} whenever $z\in \mathbb{B}^d,\eta\in\mathbb{S}^d$ and $\vert z-\eta\vert<\delta'.$ Denote $B_{1-\frac{\delta'}{2}}''=\{z\in \mathbb{C}^d: \vert z\vert \leq1-\frac{\delta'}{2}\}$ and $C_{\frac{\delta'}{2}}=\{z\in \mathbb{C}^d: 1-\frac{\delta'}{2}<\vert z\vert \leq1\}.$ Then the closed ball $\overline{\mathbb{B}^d}$ has the following decomposition,  \begin{equation}\label{dec}\overline{\mathbb{B}^d}=B_{1-\frac{\delta'}{2}}''\cup C_{\frac{\delta'}{2}} \text{~and~} B_{1-\frac{\delta'}{2}}''\cap C_{\frac{\delta'}{2}} =\emptyset.\notag\end{equation} Since the function  $\frac{1}{\vert 1-\langle z,w\rangle\vert^\alpha}$  is uniformly continuous on the compact set $B_{1-\frac{\delta'}{2}}''\times\overline{\mathbb{B}^d},$  there exists a $0<\delta''<1$ such that \begin{equation}\label{limw}\left\vert\frac{1}{\vert 1-\langle z_1,w\rangle\vert^\alpha}-\frac{1}{\vert 1-\langle z_2,w\rangle\vert^\alpha}\right\vert\leq\varepsilon,\end{equation}     whenever $(z_1,w),(z_2,w)\in B_{1-\frac{\delta'}{2}}''\times\overline{\mathbb{B}^d}$ and $\vert z_1-z_2\vert<\delta''.$ Take $\delta'''=\min\{\frac{\delta'}{2},\delta''\}.$ Now we prove that, for any $z_1,z_2\in \overline{\mathbb{B}^d}$ such that $\vert z_1-z_2\vert<\delta''',$ then
 \begin{equation}\label{limepp}\int_{\mathbb{B}^d}\left\vert\frac{1}{\vert 1-\langle z_1,w\rangle\vert ^\alpha}-\frac{1}{\vert 1-\langle z_2,w\rangle\vert^\alpha}\right\vert dv(w)\leq\varepsilon.\end{equation}
In fact, there are two cases need to be considered.  The first case is $z_1\in C_{\frac{\delta'}{2}}$ or  $z_2\in C_{\frac{\delta'}{2}}.$ Without loss of generality, we can assume that $z_1\in C_{\frac{\delta'}{2}},$ then there exists
an $\eta \in \mathbb{S}^d$ satisfying that $\vert z_1-\eta\vert<\delta'''\leq\frac{\delta'}{2}.$ Thus, obviously, the triangle inequality implies that $\vert z_2-\eta\vert\leq\vert z_2-z_1\vert +\vert z_1-\eta\vert<\delta'.$ This  together with (\ref{limee}) implies that \begin{equation}\begin{split}\label{limeef}&\int_{\mathbb{B}^d}\left\vert\frac{1}{\vert1-\langle z_1,w\rangle\vert^\alpha}-\frac{1}{\vert 1-\langle z_2,w\rangle\vert^\alpha}\right\vert dv(w)\notag\\
&\leq\int_{\mathbb{B}^d}\left\vert\frac{1}{\vert 1-\langle z_1,w\rangle\vert^\alpha}-\frac{1}{\vert 1-\langle \eta,w\rangle\vert^\alpha}\right\vert dv(w)\\&~\vspace{0.2cm}+\int_{\mathbb{B}^d}\left\vert\frac{1}{\vert 1-\langle \eta, w\rangle\vert^\alpha}-\frac{1}{\vert 1-\langle z_2,w\rangle\vert^\alpha}\right\vert dv(w)\\&\leq \varepsilon
\end{split}\end{equation} The second case is $z_1,z_2\in B_{1-\frac{\delta'}{2}}''.$ From (\ref{limw}), it implies that \begin{equation}\begin{split}\label{limeeff}\int_{\mathbb{B}^d}\left\vert\frac{1}{\vert 1-\langle z_1,w\rangle\vert^\alpha}-\frac{1}{\vert  1-\langle z_2,w\rangle\vert^\alpha}\right\vert dv(w)\leq \varepsilon  \int_{\mathbb{B}^d}dv= \varepsilon.\notag\end{split}\end{equation}
It proves (\ref{limepp}). Combing (\ref{limepp})  with $$\vert K_\alpha^+ f_n(z_1)-K_\alpha^+ f_n(z_2)\vert\leq\Vert f_n\Vert_\infty\int_{\mathbb{B}^d}\left\vert\frac{1}{\vert 1-\langle z_1,w\rangle\vert^\alpha}-\frac{1}{\vert 1-\langle z_2,w\rangle\vert^\alpha}\right\vert dv(w)$$  follows that $\{K_\alpha^+ f_n\}$ is equicontinuous. Then Arzel\`a-Ascoli theorem implies that $\{K_\alpha^+ f_n\}$  has a convergency subsequence in the supremum norm (or is precompact). That finishes the proof.
\end{proof}

\begin{cor}\label{kcom1} If $0<\alpha<d+1,$ then  $K_\alpha^+: L^p\rightarrow L^1$ is compact for any $1\leq p\leq \infty.$

\end{cor}
\begin{proof}  It follows from Lemma \ref{kcom} and the fact that $K_\alpha^+ $ is adjoint.
\end{proof}

Now, recall  the definition of the fractional radial differential operator $R^{s,t}$ on $\textrm{Hol}(\mathbb{B}^d).$ For any two real parameters $s$ and $t$ with the property that neither $d + s $ nor $d + s + t$ is a negative integer, the invertible operator $R^{s,t}$ is given by
  $$R^{s,t}f(z)=\sum_{n=0}^{\infty}\frac{\Gamma(d+1+s)\Gamma(d+1+n+s+t)}{\Gamma(d+1+s+t)\Gamma(d+1+n+s)}f_n(z),$$ for any $f=\sum_{n=0}^{\infty}f_n\in \textrm{Hol}(\mathbb{B}^d)$  with homogeneous expansion.  From   \cite[Proposition 1.2.6]{Rud}, we know  every holomorphic function $f\in \textrm{Hol}(\mathbb{B}^d)$ has a global power series expansion, thus the definition is well-defined on $\textrm{Hol}(\mathbb{B}^d).$ In fact, it can be checked by the direct calculation that the invertible operator of  $R^{s,t}$ is just $R^{s+t,-t}.$ Be careful that  the  invertible operator here is unnecessarily continuous. Recall that   $A^p=\textrm{Hol}(\mathbb{B}^d)\cap L^p$ is the $p$-integrable  Bergman space on $\mathbb{B}^d$  for $1\leq p\leq \infty.$

 \begin{lem}\label{kr} Suppose $\alpha\in\mathbb{R}$ satisfying $\alpha$ is not a nonpositive integer and $1\leq p\leq \infty,$  then the following holds on $A^p$ $$K_\alpha=R^{0,\alpha-d-1}. $$
 \end{lem}
  \begin{proof} Since every holomorphic function $f\in  \textrm{Hol}(\mathbb{B}^d)$ has a global power series expansion, we can suppose $f=\sum_{n=0}^{\infty}f_n\in A^p$ with the homogeneous    expansion.  Then the dominated convergence theorem  and formula (1.21) in \cite{Zhu} imply that 
     \begin{equation}\begin{split}\label{kass}
     K_\alpha f(z)&=\int _{\mathbb{B}^d}\frac{f(w)}{(1-\langle z,w\rangle)^\alpha}dv(w)\\
                         &=\sum_{n=0}^{\infty} \frac{\Gamma(n+\alpha)}{\Gamma(n+1)\Gamma(\alpha)}\int _{\mathbb{B}^d}f(w)\langle z,w\rangle^ndv(w)\\
                          &=\sum_{n=0}^{\infty} \frac{\Gamma(n+\alpha)}{\Gamma(n+1)\Gamma(\alpha)}2d\int _0^1r^{2d-1}dr \sum_{k=0}^\infty r^{k+n} \int_{\mathbb{S}^d}f_k(\xi)\langle z,\xi \rangle^n d\sigma(\xi)\\
                          &=\sum_{n=0}^{\infty} \frac{\Gamma(n+\alpha)}{\Gamma(n+1)\Gamma(\alpha)}2d\int _0^1r^{2d-1}dr  \int_{\mathbb{S}^d}f_n(r\xi)\langle z,r\xi \rangle^n d\sigma(\xi)\\
                           &=\sum_{n=0}^{\infty} \frac{\Gamma(n+\alpha)}{\Gamma(n+1)\Gamma(\alpha)}\int f_n(w) \langle z,w \rangle^n dv(w)\\
                          &= \sum_{n=0}^{\infty} \frac{\Gamma(d+1)\Gamma(\alpha+n)}{\Gamma(\alpha)\Gamma(d+1+n)}f_n(z),
 \end{split}\end{equation}
 for any $z\in\mathbb{B}^d.$
    It leads to the desired result.
\end{proof}

    \begin{lem}\cite[Proposition 5]{ZZ}\label{ZhZ} Suppose $s, t$ are real parameters such that neither $n + s$ nor $n + s + t$ is a negative integer. Then, for any nonnegative integer $N,$ $$R^{s,t} \frac{1}{( 1-\langle z,w\rangle)^{d+1+s+N}}= \frac{h(\langle z,w\rangle) } {( 1-\langle z,w\rangle)^{d+1+s+N+t}},$$
where $h$ is a certain one-variable polynomial of degree $N$. In particular, $h\equiv 1$ if $N=0.$
\end{lem}

   \begin{prop}\label{kfcpt} Suppose $0<\alpha\leq d+1$ and $1<p, q<\infty,$ if $K_\alpha^+:L^p \rightarrow L^q$ is compact, then  $\frac{1}{q}> \frac{1}{p}+ \frac{\alpha}{d+1}-1.$
\end{prop}
\begin{proof} For every $z\in \mathbb{B}^d,$ denote the holomorphic function $H_z$ on $\mathbb{B}^d$ by \begin{equation}\label{Hz} H_z(w)=\frac{(1-\vert z\vert^2)^{\frac{s}{q}}}{(1-\langle w,z\rangle)^{\frac{s}{q}+\frac{d+1}{p}+\alpha-d-1}},\quad w\in \mathbb{B}^d,\end{equation}
where $s$ is a positive parameter satisfying  $\frac{s}{q}-(1-\frac{1}{p})(d+1)$ is a large enough positive integer. Clearly $H_z$ is bounded holomorphic on $\mathbb{B}^d$ for each $z\in\mathbb{B}^d.$
We first prove  $\vert R^{\alpha-d-1,d+1-\alpha}H_z\vert \rightarrow 0 $ weakly in $L^p,$ as $\vert z\vert\rightarrow 1^-.$ By    \cite[Section 8.3.3, Theorem 2]{AK}, it suffices to prove that \begin{enumerate}
\item     $ \sup_{z\in \mathbb{B}^d}  \Vert R^{\alpha-d-1,d+1-\alpha}H_z\Vert_p<\infty;$
\item $ \lim_{\vert z\vert\rightarrow 1^-}\int_{\mathbb{B}^d}\vert R^{\alpha-d-1,d+1-\alpha}H_z\vert dv=0.$
\end{enumerate}
It is immediate from Lemma \ref{ZhZ} that  \begin{equation}\label{RH} R^{\alpha-d-1,d+1-\alpha}H_z(w)=\frac{h_{s}(\langle w,z\rangle)(1-\vert z\vert^2)^{\frac{s}{q}}}{(1-\langle w,z\rangle)^{\frac{s}{q}+\frac{d+1}{p}}},\quad w\in \mathbb{B}^d,\end{equation}
where $h_s$ is a certain one-variable polynomial of degree $\frac{s}{q}-(1-\frac{1}{p})(d+1).$
 From  Lemma \ref{FR}, there exists a positive constant C such that  $$\int_{\mathbb{B}^d} \frac{\vert h_{s}(\langle w,z\rangle)\vert dv(w)}{\vert 1-\langle w,z\rangle\vert^{\frac{s}{q}+\frac{d+1}{p}}}\leq C(1-\vert z\vert^2)^{(1-\frac{1}{p})(d+1)-\frac{s}{q}},\quad z\in \mathbb{B}^d.$$
Then by (\ref{RH}), the following estimate holds,
 \begin{equation}\begin{split}\label{}
  \int_{\mathbb{B}^d}\vert R^{\alpha-d-1,d+1-\alpha}H_z\vert dv&\leq \int_{\mathbb{B}^d} \frac{\vert h_{s}(\langle w,z\rangle)\vert (1-\vert z\vert^2)^{\frac{s}{q}} dv(w)}{\vert 1-\langle w,z\rangle\vert^{\frac{s}{q}+\frac{d+1}{p}}}\notag\\
  &\leq C(1-\vert z\vert^2)^{(1-\frac{1}{p})(d+1)}.
  \end{split}\end{equation}
  Thus, $ \lim_{\vert z\vert\rightarrow 1^-}\int_{\mathbb{B}^d}\vert R^{\alpha-d-1,d+1-\alpha}H_z\vert dv=0,$ namely, condition (2) holds. Similarly, we can verify the condition (1)  by Lemma \ref{FR}. Thus we prove that $\vert R^{\alpha-d-1,d+1-\alpha}H_z\vert \rightarrow 0 $ weakly in $L^p,$ as $\vert z\vert\rightarrow 1^-.$ Now, combing with the well-known fact that a compact operator maps a weakly convergent sequence into a strongly convergent one, we obtain  $$\lim_{\vert z\vert\rightarrow 1^-} \Vert K_\alpha^+(\vert R^{\alpha-d-1,d+1-\alpha}H_z\vert)\Vert_q =0.$$
  Note that $K_\alpha^+(\vert f\vert)\geq \vert K_\alpha( f)\vert  $ for any $f\in L^1,$ then $$\lim_{\vert z\vert\rightarrow 1^-}\Vert K_\alpha(R^{\alpha-d-1,d+1-\alpha}H_z)\Vert_q =0.$$ Since $R^{0,\alpha-d-1}R^{\alpha-d-1,d+1-\alpha}f=f$ for any bounded holomorphic function $f$ on $\mathbb{B}^d,$ it follows  from Lemma \ref{kr} that $$\lim_{\vert z\vert\rightarrow 1^-}\Vert H_z\Vert_q= \lim_{\vert z\vert\rightarrow 1^-}\Vert R^{0,\alpha-d-1}R^{\alpha-d-1,d+1-\alpha}H_z\Vert_q =\lim_{\vert z\vert\rightarrow 1^-}\Vert K_\alpha R^{\alpha-d-1,d+1-\alpha}H_z\Vert_q =0,$$  and then
  $$\lim_{\vert z\vert\rightarrow 1^-}\int_{\mathbb{B}^d}\frac{(1-\vert z\vert^2)^{s}dv(w)}{\vert 1-\langle w,z\rangle\vert^{s+\frac{q(d+1)}{p}+q(\alpha-d-1)}} =\lim_{\vert z\vert\rightarrow 1^-}\Vert H_z\Vert_q^q =0.$$
  Since $s>0$ is choosen large enough, by Lemma \ref{FR} again, we  obtain that $\frac{1}{q}> \frac{1}{p}+ \frac{\alpha}{d+1}-1.$
 \end{proof}
 \begin{rem}  (1) We have proved $\vert R^{\alpha-d-1,d+1-\alpha}H_z\vert \rightarrow 0 $ weakly in $L^p,$ as $\vert z\vert\rightarrow 1^-.$ Note that bounded operators map
a weakly convergent sequence into a weakly convergent one. Analogously, we can prove that, if $K_\alpha^+:L^p \rightarrow L^q$ is bonded as the same condition in Proposition \ref{kfcpt}, then  $$\sup_{z\in \mathbb{B}^d}\int_{\mathbb{B}^d}\frac{(1-\vert z\vert^2)^{s}dv(w)}{\vert 1-\langle w,z\rangle\vert^{s+\frac{q(d+1)}{p}+q(\alpha-d-1)}} <\infty.$$   Then Lemman \ref{FR}  implies that $\frac{1}{q}\geq \frac{1}{p}+ \frac{\alpha}{d+1}-1.$ In fact, this  gives a new proof of necessity part of \cite[Lemma 4.3]{DiW}, which is a key step to  characterize the $L^p$-$L^q$ boundedness  of $K_\alpha,K_\alpha^+.$

(2) The construction of $H_z$ is inspired by  the characterization of  vanishing Carleson measure for Bergman space on the unit ball, we refer the reader to \cite{ZZ,Zhu}.
 \end{rem}

 Now we turn to the proof of Theorem 2, we first describe the $L^p $-$ L^q$ compactness of $K_{d+1}^+.$ It is as analogous to the proof of Theorem 1.
 Thus we give a sketchy proof here.

  \begin{prop}\label{kd1}  $K_{d+1}^+:L^p \rightarrow L^q$ is compact if and only if $1\leq q<p\leq\infty.$
  \end{prop}
   \begin{proof} 
    By Lemma \ref{cpt} and the adjointness of $K_\alpha^+,$ it suffices to prove $K_{d+1}^+:L^p \rightarrow L^1$ is compact for any $1< p\leq \infty.$ 
Note that  \cite[Theorem 3]{DiW} shows  $K_{d+1}^+:L^\infty \rightarrow L^1$ is bounded,  combing with Lemma \ref{cpt} again, it suffices to prove  $K_{d+1}^+:L^p \rightarrow L^1$ is compact for any $d+1< p<\infty.$ Let $\{B_j'\}$ be as in Proof of Theorem 1, and $\{f_j\}$ be a bounded sequence in $L^p.$ Then we can prove $\Vert ( K_\alpha^+ f_n)|_{B_j'}\Vert_\infty\leq j^\alpha \Vert f_n\Vert_p,$ combing with Lemma \ref{klp} and Arzel\`a-Ascoli theorem implies that $\{(K_\alpha^+ f_n)|_{B_j'}\}$ is precompact in $L^p(B_j').$ By Lemma \ref{FR} and H\"older's inequality, it implies, there exists a constant $C>0$ such that

 \begin{equation}\begin{split}\label{Kfjj}
\int_{\mathbb{B}^d-B_j'}\vert K_{d+1}^+f_n(z)\vert dv(z)
&\leq \Vert f_n\Vert_p \int_{\mathbb{B}^d-B_j'}\left\vert\int_{\mathbb{B}^d}\frac{dv(w)}{\vert 1-\langle z,w\rangle\vert^{\frac{p(d+1)}{p-1}}}\right\vert
^{1-\frac{1}{p}}dv(z)\notag\\
&\leq C \Vert f_n\Vert_p \int_{\mathbb{B}^d-B_j'} (1-\vert z\vert^2)^{-\frac{d+1}{p}} dv(z).\\
\end{split}\end{equation}
 Thus we obtain that, if  $d+1< p<\infty,$ then for any $\varepsilon>0,$ there exists a $J>0,$ such that $$\int_{\mathbb{B}^d-B_j'}\vert K_{d+1}^+f_n\vert dv\leq C\varepsilon,$$
 whenever $j>J.$ Therefore,  $\{K_\alpha^+ f_n\}$ is precompact in $L^1(\mathbb{B}^d)$ by Lemma \ref{AdF}.
 \end{proof}

 {\noindent{\bf{Proof of  Theorem 2.}} Since the equivalence of $(1)\Leftrightarrow(3)$ is the main result of  \cite[Theorem 3]{DiW},  it suffices to prove $(2)\Leftrightarrow (3).$
  There are two cases $\alpha=d+1$ and $0<\alpha<d+1$ to be considered. Indeed, the case $\alpha=d+1$ has been proved in Proposition \ref{kd1}. Thus,
 it remains to deal with the case  $0<\alpha<d+1.$ By Lemma \ref{cpt} and Corollary \ref{kcom1}, we conclude that $K_\alpha^+: L^1\rightarrow L^q$ is compact if and only if $q<\frac{d+1}{\alpha} ,$ and then $K_\alpha^+: L^p\rightarrow L^\infty$ is compact if and only if $p>\frac{d+1}{d+1-\alpha}.$ Then apply  Lemma \ref{cpt}, Proposition \ref{kcom} and Proposition \ref{kfcpt}  to obtain the  equivalence of (2) and (3) when  $0<\alpha<d+1.$ It completes the proof. \qed


 We have completely characterized that $L^p$-$L^q$ compactness of the Bergman-type operators $K_\alpha,K_\alpha^+$ when $0<\alpha<d+2$ so far.  However, when $\alpha\geq d+2,$  \cite[Theorem 4]{DiW} shows that there exist no $1\leq p,q\leq \infty $ such that $K_\alpha^+:L^p\rightarrow L^q$ is bounded.

To end this section, we describe the phenomenon for $\alpha\leq0.$ Recall that a bounded  operator between two Banach spaces is called a finite rank operator if the range of the operator has finite dimension.
Obviously,  finite rank operators must be compact and the finite rank operators on a Hilbert space belong to every Schatten $p$-class with $0 < p<\infty.$
 \begin{prop}\label{fink} Suppose $\alpha\leq0,$ then the followings hold.
 \begin{enumerate}
\item $K_\alpha,K_\alpha^+:L^p\rightarrow L^q$ are compact for any $1\leq p,q\leq \infty .$
\item If $\alpha$ is a  nonpositive integer, then  $K_\alpha:L^p\rightarrow L^q$ is  a  finite rank operator for any $1\leq p,q\leq \infty .$
\item  If  $\alpha$ is a nonpositive even  integer, then $K_\alpha^+:L^p\rightarrow L^q$ is  a finite rank operator for any $1\leq p,q\leq \infty .$
\end{enumerate}
 \end{prop}
 \begin{proof}(1) It suffices to prove that $K_\alpha,K_\alpha^+:L^1\rightarrow L^\infty$ are compact when $\alpha\leq0.$ Suppose $\{f_j\}$ is  an arbitrary bounded sequence in $L^\infty,$ without loss of generality, we can suppose that
   \begin{equation}\label{bdss} \Vert f_j\Vert_\infty \leq 1,\quad j=1,2,\cdots.\notag\\ \end{equation} Then Lemma \ref{klp} implies that $\{K_\alpha f_j\}$  is  continuous function sequences on $\mathbb{B}^d.$ Note that the kernel function $\frac{1}{(1-\langle z,w\rangle)^\alpha}$ is  uniformly continuous on the compact set $\overline{\mathbb{B}^d}\times\overline{\mathbb{B}^d}$ when $\alpha\leq0,$
 then  we obtain that $\{K_\alpha f_j\}$ is  in fact continuous function sequence on  $\overline{\mathbb{B}^d}.$
 Since $\alpha\leq0,$ it follows that
 \begin{equation}\begin{split}\label{Kfja}
\Vert  K_\alpha f_n\Vert_\infty
&=\sup_{z\in \overline{\mathbb{B}^d}}\left\vert\int_{\mathbb{B}^d}\frac{f_n(w)dv(w)}{( 1-\langle z,w\rangle)^\alpha}\right\vert\notag\\
 &\leq 2^{-\alpha}\Vert f_n\Vert_1\\
  &\leq 2^{-\alpha},
\end{split}\end{equation}
Thus  $\{K_\alpha f_j\}$  is a bounded subset in  $C(\overline{\mathbb{B}^d}).$
The uniform continuity of the function $\frac{1}{(1-\langle z,w \rangle)^\alpha}$ also yields that,  for any $\varepsilon>0,$ there exists a $\delta>0,$ such that
$$\left\vert\frac{1}{( 1-\langle z_1,w \rangle)^\alpha}-\frac{1}{(1-\langle z_2,w \rangle)^\alpha}\right\vert\leq\varepsilon,$$
whenever $z_1,z_2\in \overline{\mathbb{B}^d}$ and $\vert z_1-z_2\vert<\delta.$ Then we have
\begin{equation}\begin{split}
\left\vert K_\alpha f_j(z_1)- K_\alpha f_j(z_2)\right\vert&\leq\int _{\mathbb{B}^d}\vert f_j\vert\left\vert\frac{1}{( 1-\langle z_1,w \rangle)^\alpha}-\frac{1}{(1-\langle z_2,w \rangle)^\alpha}\right\vert dv(w)\notag\\
&\leq \Vert f_j\Vert_1\cdot\varepsilon\\
&\leq\varepsilon
\end{split}\end{equation}
whenever $z_1,z_2\in\overline{ \mathbb{B}^d}$ and $\vert z_1-z_2\vert<\delta.$ That means  $\{K_\alpha f_j\}$  is equicontinuous. Then Arzel\`a-Ascoli theorem implies that $\{K_\alpha f_j\}$  has a convergency subsequence in the supremum norm (or is precompact). That  proves $K_\alpha :L^1\rightarrow L^\infty$ is  compact when $\alpha\leq0.$ Similarly,  $K_\alpha^+ :L^1\rightarrow L^\infty$ is  compact when $\alpha\leq0.$

(2)  It suffices to prove that $K_\alpha:L^1\rightarrow L^\infty$ is a finite rank operator  when $\alpha$ a  nonpositive integer. Suppose that $\alpha=-N,$ where $N$ is a nonnegative integer. Then  the binomial theorem implies that $$(1-\langle z,w \rangle)^N=\sum_{\vert s\vert\leq N}a_{N,s}z^s\bar{w}^s,$$
where $s\geq 0$ is multi-index and $a_{N,s}$ is nonzero constant. Thus, for any $f\in L^1,$ we have $$K_{-N} f(z)=\int_{\mathbb{B}^d}f(w)( 1-\langle z,w\rangle)^Ndv=\sum_{\vert s\vert\leq N}\left(a_{N,s}\int_{\mathbb{B}^d}f\bar{w}^sdv\right)z^s.$$ It implies that $K_{-N}$ is a finite rank operator.

(3)  It suffices to prove that $K_\alpha^+:L^1\rightarrow L^\infty$ is a finite rank operator  when $\alpha$ a   nonpositive even integer. Suppose that $\alpha=-2N,$ where $N$ is a nonnegative integer. Then $$\vert 1-\langle z,w \rangle\vert^{2N}=(1-\langle z,w \rangle)^N(1-\langle w,z \rangle)^N=\sum_{\vert s\vert,\vert l\vert\leq N}a_{N,s}a_{N,l}z^s\bar{z}^lw^l\bar{w}^s.$$
Thus, for any $f\in L^1,$ we have $$K_{-2N} f(z)=\int_{\mathbb{B}^d}f\vert 1-\langle z,w\rangle\vert^{2N}dv=\sum_{\vert s\vert,\vert l\vert\leq N}\left(a_{N,s}a_{N,l}\int_{\mathbb{B}^d}fw^l\bar{w}^sdv\right)z^s\bar{z}^l.$$ It implies that $K_{-2N}$ is a finite rank operator.
 \end{proof}
 \begin{rem} $K_{\alpha}^+$ will  not be a finite rank operator when $\alpha$ is a nonpositive odd integer.
 \end{rem}
 Thus we have the following corollary.
\begin{cor}\label{kcor} For  $\alpha\in\mathbb{R},$ then the following statements are equivalent:
\begin{enumerate}
\item $\alpha< d+2;$
\item there exist $1\leq p,q\leq \infty $ such that $K_\alpha:L^p\rightarrow L^q$ is bounded;
\item there exist $1\leq p,q\leq \infty $ such that $K_\alpha^+:L^p\rightarrow L^q$ is bounded;
\item there exist $1\leq p,q\leq \infty $ such that $K_\alpha:L^p\rightarrow L^q$ is compact;
\item there exist $1\leq p,q\leq \infty $ such that $K_\alpha^+:L^p\rightarrow L^q$ is compact.
\end{enumerate}
 \end{cor}

 \section{Schatten class and Macaev class membership}\label{pfs}
 In the present section, we complete proofs of  Theorem 3,4, which give the necessary and sufficient conditions that ensure the  Bergman-type operator $K_\alpha$ belongs to Schatten classes and Macaev classes. The proof of Theorem 3 will be respectively given in  cases $0<\alpha<d+1$ and  $\alpha\leq0.$  Although  Theorem 3 can be uniformly  proved by the method of estimates of the spectrum, we  prefer to prove the theorem in the case  $0<\alpha<d+1$ by using the method of operator theory which  is inspired by the characterization of Schatten class Toeplitz operators and Hankel operators in \cite{Pa,Zhu1,Zhu2}.

 Denote the point spectrum (the collections of eigenvalues) of $K_\alpha $ on  the Bergman space $A^2$ by $\sigma_{pt}(K_\alpha,A^2).$

 \begin{lem}\label{spec} Suppose $\alpha<d+1,$ then   the followings hold.
  \begin{enumerate}
  \item If $\alpha\in\mathscr{S}_d,$ then $\sigma_{pt}(K_\alpha,A^2)=\{\frac{\Gamma(d+1)\Gamma(\alpha+n)}{\Gamma(\alpha)\Gamma(d+1+n)}:n\in\mathbb{N}\}.$
  \item If $-\alpha\in\mathbb{N},$  then $\sigma_{pt}(K_\alpha,A^2)=\{0\}\cup\{\frac{(-1)^n\Gamma(1-\alpha)\Gamma(d+1)}{\Gamma(1-\alpha-n)\Gamma(n+d+1)}: 0\leq n\leq -\alpha\}.$
\end{enumerate}
   \end{lem}
   \begin{proof} (1) Denote $\mu_n=\frac{\Gamma(d+1)\Gamma(\alpha+n)}{\Gamma(\alpha)\Gamma(d+1+n)},n\in\mathbb{N}.$ Due to Lemma \ref{kr}, it suffices to show that  $\sigma_{pt}(K_\alpha,A^2)\subset\{\mu_n:n\in\mathbb{N}\}.$ Suppose $\mu\in\sigma_{pt}(K_\alpha,A^2),$ then there exists a nonzero $f\in A^2$ such that \begin{equation}\label{cha} K_\alpha f=\mu f.\end{equation}
 It is easy to see that $\cup_{n\in\mathbb{N}} P_{n}$ is a orthonormal basis of $A^2,$  where $P_{n}$ is  the set of homogeneous polynomials defined by $$P_{n}=\left\{c_kz^k: \sum_{j=1}^d k_j=n, k_j \in \mathbb{N}\right\},\quad  n\in\mathbb{N},$$ and $c_k$ is the normalized positive constant such that $\Vert c_kz^k\Vert_2=1$ for each $k.$ Thus $f$ has the following representation $$f=\sum_{n=0}^\infty\sum_{e_{n,k}\in P_{n}}\langle f,e_{n,k}\rangle e_{n,k},$$  where $ \langle \cdot,\cdot \rangle$ is the standard Hermitian inner product on $A^2.$  Combing this with Lemma \ref{kr} and (\ref{cha}), we conclude that
\begin{equation}\label{decof}\sum_{n=0}^\infty(\mu-\mu_n)\sum_{e_{n,k}\in P_{n}}\langle f,e_{n,k}\rangle e_{n,k}=0.\end{equation} Since $f$ is nonzero, there exists a $e_{n_0,k_0}$ such that $\langle f,e_{n_0,k_0}\rangle\neq0.$ Then (\ref{decof}) implies that $f$ has the form \begin{equation}\label{cheq} f=\sum_{e_{n,k}\in P_{n_0}}\langle f,e_{n,k}\rangle e_{n,k},\end{equation} and $\mu=\mu_{n_0}.$ It completes the proof of $\sigma_{pt}(K_\alpha,A^2)\subset\{\mu_n:n\in\mathbb{N}\}.$

(2) Suppose $f=\sum f_n\in A^2$ with  the homogeneous expansion. Since  $ -\alpha$ is a nature number, it follows that
 \begin{equation}\begin{split}\label{}
 K_\alpha f(z)&=\int_{\mathbb{B}^d} \sum f_n(w) (1-\langle z,w\rangle)^{-\alpha}dv(w)\notag\\
                     &=\int_{\mathbb{B}^d}  \sum_{n=0}^\infty f_n \sum_{k=0}^{-\alpha} \frac{(-\alpha)!(-1)^k}{(-\alpha-k)!k!} \langle z,w\rangle^{k}dv(w)\\
                     &=\sum_{n=0}^{-\alpha} \frac{(-\alpha)!(-1)^n}{(-\alpha-n)!n!}\int f_n(w) \langle z,w\rangle^{n}dv(w)\\
                  &=\sum_{n=0}^{-\alpha} \frac{(-1)^n\Gamma(1-\alpha)\Gamma(d+1)}{\Gamma(1-\alpha-n)\Gamma(n+d+1)}f_n(z)
 \end{split}\end{equation}
 for any $z\in\mathbb{B}^d.$ Then, the same argument as above shows that  $$\sigma_{pt}(K_\alpha,A^2)=\{0\}\cup\left\{\frac{(-1)^n\Gamma(1-\alpha)\Gamma(d+1)}{\Gamma(1-\alpha-n)\Gamma(n+d+1)}: 0\leq n\leq -\alpha\right\}.$$
   \end{proof}
 Theorem 2, Lemma \ref{klp}, Lemma \ref{kr} and the boundedness of embedding $A^2\rightarrow L^2,$ imply that $K_\alpha$ is a normal compact operator on the Bergman space $A^2$ when $\alpha<d+1.$ In particular, $K_\alpha$  is normal compact  on  $A^2$ when $0<\alpha<d+1,$ thus we can apply the functional calculation to $K_\alpha.$ Set function $F_p(x)=x^p$ on $\mathbb{R}_{\geq0}$ for any $p>0.$ Note that $F_p:\mathbb{R}_{\geq0}\rightarrow\mathbb{R}_{\geq0}$ is bijective. We define $K_\alpha^p=F_p(K_\alpha),$ which is the functional calculation of  $K_\alpha\in B(A^2)$  with respect to the function $F_p(x)=x^p,$ where  $B(A^2)$ is the collections of bounded operators on the Bergman space $A^2.$

 \begin{lem}\label{funo} Suppose $0<\alpha<d+1.$ Then for any $0<p<\frac{d+1}{d+1-\alpha},$ there exists a positive constant $C_p$ such that the following operator inequality
      $$\frac{1}{C_p} K_{p\alpha-(p-1)(d+1)}\leq K_\alpha^p\leq C_pK_{p\alpha-(p-1)(d+1)}$$ holds on the Bergman space $A^2.$
 \end{lem}
 \begin{proof} Since $K_\alpha$ is  positive and compact on $A^2,$ it follows from \cite[Theorem 1.9.2]{Rin} that  $K_\alpha$ admits the following  canonical decomposition
  $$  K_\alpha f=\sum_{n=0}^{\infty}\lambda_n \langle f,e_{n}\rangle e_{n},$$ where $\{\lambda_n\}$ is the sequence of nonzero eigenvalues (counting multiplicities) with decreasing order, $\{e_n\}$ is the corresponding  orthonormal sequence of eigenvectors and $ \langle \cdot,\cdot \rangle$ is the standard Hermitian inner product on $A^2.$
 Combing with  Lemma \ref{spec}, we can further suppose \begin{equation}\label{cade} K_\alpha f=\sum_{n=0}^{\infty}\frac{\Gamma(d+1)\Gamma(\alpha+n)}{\Gamma(\alpha)\Gamma(d+1+n)}\sum_{e_{n,k}\in P_{n}}\langle f,e_{n,k}\rangle e_{n,k}  \end{equation} is the canonical decomposition of $K_\alpha.$   It implies from the  functional calculation and (\ref{cade}) that  $$ K_\alpha^p f=\sum_{n=0}^{\infty}\left( \frac{\Gamma(d+1)\Gamma(\alpha+n)}{\Gamma(\alpha)\Gamma(d+1+n)}\right)^p\sum_{e_{n,k}\in P_{n}}\langle f,e_{n,k}\rangle e_{n,k},  $$
 is the canonical decomposition of $K_\alpha^p.$  Note that the condition $0<p<\frac{d+1}{d+1-\alpha}$ ensures $p\alpha-(p-1)(d+1)>0.$  By Stirling's formula, we conclude that
                  $$\left( \frac{\Gamma(d+1)\Gamma(\alpha+n)}{\Gamma(\alpha)\Gamma(d+1+n)}\right)^p\sim n^{p(\alpha-(d+1))}\sim \frac{\Gamma(d+1)\Gamma(p\alpha-(p-1)(d+1)+n)}{\Gamma(p\alpha-(p-1)(d+1))\Gamma(d+1+n)}, n\rightarrow\infty.$$
Together this with Lemma \ref{kr} shows that  there exists a positive constant $C_p$ satisfying
 $$\langle\frac{1}{C_p} K_{p\alpha-(p-1)(d+1)}f,f\rangle\leq\langle K_\alpha^p f,f  \rangle\leq  \langle C_p K_{p\alpha-(p-1)(d+1)}f,f\rangle,$$
for any  $f\in A^2.$ This finishes the proof.
 \end{proof}

 Now we recall the Berezin transform on the unit ball $\mathbb{B}^d.$
The Bergman kernel of $\mathbb{B}^d $ is given by
$$K_w(z)= K(z,w)=\frac{1}{(1-\langle z,w\rangle)^{d +1}}, \quad   z,w \in \mathbb{B}^d ,$$
which is also called the reproducing kernel of $A^2,$ since  \begin{equation}\label{rep} f(z)=\langle f, K_z\rangle,  \quad   z \in \mathbb{B}^d\end{equation}
for any $f\in A^2.$ The normalized reproducing kernel of  $A^2$ is
 \begin{equation}\label{nrep}  k_{w}(z)=\frac{K(z,w)}{\sqrt{K(w,w)}}=\frac{(1-|w|^2)^\frac{d+1}{2}}{(1-\langle z,w\rangle)^{d+1}},  \quad   z,w \in \mathbb{B}^d.\end{equation}
For a bounded operator $T\in B(A^2)$, the Berezin transform $ \widetilde{T}$ of $T$ is given by
$$ \widetilde{T}(z)=\langle T k_{z},k_{z}\rangle,\quad z \in \mathbb{B}^d.$$
The M\"obius  invariant measure $d\lambda$ on $\mathbb{B}^d $ is defined by $$ d\lambda(z)=\frac{dv(z)}{(1-\vert z\vert^2)^{d+1}}.$$
The Berezin transform  is an important tool in the operator theory on the holomorphic function space, see \cite{Zhu,Zhu1,Zhu2} for more details.
In what follows, we calculate the Berezin transform of $K_\alpha.$
\begin{lem}\label{bere} $\widetilde{K_\alpha}(z)=(1-\vert z\vert^2)^{d+1-\alpha}.$
\end{lem}
  \begin{proof} For every $w\in \mathbb{B}^d,$ denote the holomorphic function $K_{\alpha,w}(z)$ on $\mathbb{B}^d$ by \begin{equation}\label{kaz} K_{\alpha,w}(z)=\frac{1}{(1-\langle z,w\rangle)^{\alpha}}.\notag\\ \end{equation}
  Obviously,  $K_{\alpha,w}\in A^2$ for every $w\in \mathbb{B}^d.$
  By (\ref{rep}), we have
  $$ K_\alpha K_z(w) =\overline{\langle K_{\alpha,w},K_z\rangle}=\frac{1}{(1-\langle w,z\rangle)^{\alpha}} .$$
Combing this with (\ref{rep}) and  (\ref{nrep}),   we get that
 \begin{equation}\begin{split}\label{}
 \widetilde{K_\alpha}(z)&= \langle K_\alpha k_z,k_z\rangle \notag\\
 &=(1-\vert z\vert^2)^{d+1}\langle K_\alpha K_z,K_z\rangle \\
 &=(1-\vert z\vert^2)^{d+1}\langle K_{\alpha,z}, K_z\rangle \\
 &=(1-\vert z\vert^2)^{d+1-\alpha}.
   \end{split}\end{equation} This completes the proof.
   \end{proof}

The following  lemma establishes a connection between the Berezin transform and Schatten $p$-class on  the Bergman space $A^2,$ see \cite[Lemma C]{Pa} and \cite[Lemma  7.10]{Zhu1} for more details.
  \begin{lem}\cite{Pa,Zhu1}\label{sch} If $T\in B(A^2)$ is a positive operator, then the followings hold.
  \begin{enumerate}
   \item  $T\in S_1(A^2)$ if and only if $ \widetilde{T}\in L^1(d\lambda).$ Moreover, the following trace formula holds,
 \begin{equation}\label{tra} \textrm{Tr}(T)=\int_{\mathbb{B}^d}  \widetilde{T}d\lambda.\end{equation}
 \item For $1<p<\infty,$ then $ \widetilde{T}\in L^p(d\lambda)$ if  $T\in \mathcal{L}^p(A^2).$
   \end{enumerate}
\end{lem}
 Now we can prove Theorem 3. The proof will be given in cases $0<\alpha<d+1$ and $\alpha\leq0,$ respectively.

{\noindent{\bf{Proof of Theorem 3 for the case $0<\alpha<d+1$.}}
Note that the compact operator $K_\alpha$ is adjoint on $L^2$  by Fubini's theorem,  it follows from \cite[Theorem 1.9.2]{Rin} that $K_\alpha $ on $L^2$  admits  the canonical decomposition  \begin{equation}\label{can} K_\alpha f=\sum_{n=0}^\infty \lambda_n \langle f,e_n\rangle e_n\end{equation}
 whenever $\alpha<d+1,$ where $\{\lambda_n\}$ is the sequence of nonzero eigenvalues (counting multiplicities) and  $\{e_n\}$ is the corresponding  orthonormal sequence of eigenvectors.
It follows from  (\ref{can}) that $$ K_\alpha e_n=\lambda_n e_n,n=0,1,\cdots.$$ Since $K_\alpha e_n$ is holomorphic by Lemma \ref{klp}, we obtain that $e_n$ is  holomorphic for any integer $n\geq 0.$ Thus $K_\alpha $ on $A^2$ and  $L^2$ own the same canonical decomposition (\ref{can}). This completes the proof that (2) is always equivalent to (1).

Now we turn to prove that (2) implies (3). Namely, we need to prove   $\widetilde{K}_{\alpha}\in L^p(d\lambda)$  if  $K_\alpha\in \mathcal{L}^p(A^2).$ Suppose that  $K_\alpha\in \mathcal{L}^p(A^2).$  Note that $\frac{d+1}{d+1-\alpha}>1,$ then  Lemma \ref{sch} shows that $\widetilde{K}_{\alpha}\in L^p(d\lambda)$  if $p\geq\frac{d+1}{d+1-\alpha}.$ Thus, it suffices to consider the case $0<p<\frac{d+1}{d+1-\alpha}.$  Observe that $0<p<\frac{d+1}{d+1-\alpha}$ ensures  $p\alpha-(p-1)(d+1)>0.$ Then, by  Lemma \ref{funo},
there exists a positive constant $C_p'$ such that \begin{equation}\label{kkcp}\frac{1}{C_p'} K_{p\alpha-(p-1)(d+1)}\leq K_\alpha^p. \end{equation} Note that $K_\alpha^p\in  \mathcal{L}^1(A^2),$ together with (\ref{kkcp}) shows that $ K_{p\alpha-(p-1)(d+1)}\in  \mathcal{L}^1(A^2).$ Then  Lemma \ref{bere} and Lemma \ref{sch} imply
\begin{equation}\begin{split}\label{}
\int_{\mathbb{B}^d}  \vert \widetilde{K}_{\alpha}\vert^p d\lambda&=\int_{\mathbb{B}^d}(1-\vert z\vert^2)^{p(d+1-\alpha)}d\lambda\notag\\
&=\int_{\mathbb{B}^d}(1-\vert z\vert^2)^{d+1-(p\alpha-(p-1)(d+1))}d\lambda\\
&=\int_{\mathbb{B}^d}  \widetilde{K}_{p\alpha-(p-1)(d+1)}d\lambda\\
&= \textrm{Tr}(K_{p\alpha-(p-1)(d+1)})\\
&<\infty.\\
\end{split}\end{equation}
Then $\widetilde{K}_{\alpha}\in L^p(d\lambda).$ This shows that (2) implies (3).

Suppose $\widetilde{K_\alpha}\in L^p(d\lambda),$ we go to prove that $p>\frac{d}{d+1-\alpha}.$ By Lemma \ref{bere}, we have
 \begin{equation}\begin{split}\label{}
\int_{\mathbb{B}^d}\vert \widetilde{K_\alpha}\vert^pd\lambda&=\int_{\mathbb{B}^d}\vert (1-\vert z \vert^2)^{p(d+1-\alpha)}d\lambda\notag\\
&=\int_{\mathbb{B}^d} (1-\vert z \vert^2)^{(p-1)(d+1)-p\alpha}dv
\end{split}\end{equation}
Together this with the condition $\widetilde{K_\alpha}\in L^p(d\lambda)$ shows that $(1-\vert z \vert^2)^{(p-1)(d+1)-p\alpha}\in L^1(\mathbb{B}^d).$
 By the fact that, for $t\in \mathbb{R},$ $(1-\vert z\vert^2)^t \in L^1(\mathbb{B}^d)$ if and only if $t>-1,$ we conclude that $(p-1)(d+1)-p\alpha>-1.$
Then $p>\frac{d}{d+1-\alpha}.$ This shows that (3) implies (4).

Now we turn to prove that (4) implies (2), that means that we need to prove $K_\alpha\in \mathcal{L}^p(A^2)$ if  $p>\frac{d}{d+1-\alpha}.$ As mentioned above, we know that  $K_\alpha$ is compact on the Bergman space $A^2,$ i.e. $K_\alpha\in \mathcal{L}^\infty(A^2).$ Since $\frac{d+1}{d+1-\alpha}>1,$ it suffices to prove that $K_\alpha\in \mathcal{L}^p(A^2) $ when $\frac{d}{d+1-\alpha}<p<\frac{d+1}{d+1-\alpha}$ by the interpolation theorem of Schatten classes, see   \cite[Theorem 2.6]{Zhu1}. Now Lemma \ref{funo} shows that,
there exists a positive constant $C_p''$ such that
   \begin{equation}\label{kpkk} K_\alpha^p\leq C_p''K_{p\alpha-(p-1)(d+1)}\end{equation}
  The condition $\frac{d}{d+1-\alpha}<p<\frac{d+1}{d+1-\alpha}$  means that $$0<p\alpha -(p-1)(d+1)<1.$$
  Then Lemma \ref{bere} shows that
    \begin{equation}\begin{split}\label{}
\int_{\mathbb{B}^d}\vert \widetilde{K}_{p\alpha-(p-1)(d+1)}\vert d\lambda&=\int_{\mathbb{B}^d}(1-\vert z \vert^2)^{p(d+1-\alpha)}d\lambda\notag\\
&=\int_{\mathbb{B}^d} (1-\vert z \vert^2)^{(p-1)(d+1)-p\alpha}dv\\
&<\infty.
\end{split}\end{equation}
   Combing with Lemma \ref{sch} follows that $\widetilde{K}_{p\alpha-(p-1)(d+1)}\in  \mathcal{L}^1(A^2).$ Therefor, we obtain that $K_\alpha\in \mathcal{L}^p(A^2)$ by (\ref{kpkk}).
   This shows that (4) implies (2). That completes the proof.\qed

Although $K_\alpha $ on $A^2$ and  $L^2$ own the same canonical decomposition (\ref{can}) when $0<\alpha<d+1,$  the point spectrum  of $K_\alpha $ on $A^2$ and  $L^2$ differ by the element $0.$ Indeed, the point spectrum of $K_\alpha $ on $L^2$ is $\{0\}\cup\{\frac{\Gamma(d+1)\Gamma(\alpha+n)}{\Gamma(\alpha)\Gamma(d+1+n)}:n\in\mathbb{N}\}.$


{\noindent{\bf{Proof of Theorem 3 for  the case $\alpha\leq0$.}} We have proven that  (1) and (2) is always  equivalent when $\alpha< d+1.$ Thus, it suffice to prove (2), (3) and (4) are equivalent if  $\alpha\leq0$ and $\alpha$ is not a integer.

Suppose $\alpha\leq0$ and $\alpha$ is not a integer. From Lemma \ref{spec} we know that $\{\mu_n:n\in\mathbb{N}\} $ is exactly the point spectrum of $K_\alpha$ on $A^2,$ where $\mu_n=\frac{\Gamma(d+1)\Gamma(\alpha+n)}{\Gamma(\alpha)\Gamma(d+1+n)},n\in\mathbb{N}.$ Let $E_n$ be  the eigenspace  corresponding   to $\mu_n.$
It follows from (\ref{cheq}) that \begin{equation}\label{dimn}dim \hspace{0.3mm}E_n=\# P_n= \frac{(n+1)_{d-1}}{(d-1)!}, \quad n\in\mathbb{N}.\end{equation}
Combing this with  the definition of $\mathcal{L}^p(A^2),$  we know that    $K_\alpha \in \mathcal{L}^p(A^2)$ if and only if \begin{equation}\label{scE}\sum_{n=0}^\infty \vert\mu_n\vert^p\hspace{0.4mm} dim \hspace{0.3mm}E_n=\sum_{n=0}^\infty\left\vert\frac{\Gamma(d+1)\Gamma(\alpha+n)}{\Gamma(\alpha)\Gamma(d+1+n)}\right\vert^p \frac{(n+1)_{d-1}}{(d-1)!}<\infty.\end{equation}
Then Stirling's formula implies that (\ref{scE}) is equivalent to $$\sum_{n=1}^\infty\frac{1}{n^{p(d+1-\alpha)-(d-1)}}<\infty.$$
It follows that  $K_\alpha \in \mathcal{L}^p(A^2)$ if and only if $p>\frac{d}{d+1-\alpha}.$ This implies that (2) and (4) are equivalent if  $\alpha\leq0$ and $\alpha$ is not a integer.

Suppose $\alpha\leq0$ and $\alpha$ is not a integer.  From  Lemma \ref{bere},  we obtain that  $ \widetilde{K}_{\alpha}\in L^p(d\lambda)$ if and only if
\begin{equation}\begin{split}\label{brep}
\int_{\mathbb{B}^d}  \vert \widetilde{K}_{\alpha}\vert^p d\lambda=\int_{\mathbb{B}^d}(1-\vert z\vert^2)^{p(d+1-\alpha)}d\lambda
=\int_{\mathbb{B}^d}(1-\vert z\vert^2)^{p(d+1-\alpha)-(d+1)}dv
<\infty.\\
\end{split}\end{equation}
By the well-known fact that  $(1-\vert z\vert^2)^t \in L^1(\mathbb{B}^d)$ if and only if $t>-1,$ we conclude that (\ref{brep}) is  equivalent to  $p>\frac{d}{d+1-\alpha}.$
This implies that (3) and (4) are equivalent if  $\alpha\leq0$ and $\alpha$ is not a integer. \qed

Now we turn to prove Theorem 4. 

{\noindent{\bf{Proof of  Theorem 4.}}
It follows from Lemma \ref{klp} and the canonical decomposition (\ref{can}) when $\alpha<d+1$ that the normal  compact operator $K_\alpha$ on $L^2$ and $A^2$ has the same nonzero  eigenvalues (counting
multiplicities). We thus obtain that $K_\alpha\in \mathcal{L}^{p,\infty}(L^2)$ if and only if $K_\alpha\in \mathcal{L}^{p,\infty}(A^2)$ when $\alpha\in\mathscr{S}_d.$
This implies that (1) and (2) are equivalent. Now we turn to prove that (2) and (3) are equivalent.  From Lemma \ref{spec} we know that $\{\mu_n:n\in\mathbb{N}\} $ is exactly the point spectrum of $K_\alpha$ on $A^2,$ where $\mu_n=\frac{\Gamma(d+1)\Gamma(\alpha+n)}{\Gamma(\alpha)\Gamma(d+1+n)},n\in\mathbb{N}.$  By the functional calculus in $C^\ast$-algebra and the canonical decomposition (\ref{cade}), it implies that \begin{equation}\label{cdec}\vert K_\alpha\vert f= \sum_{n=0}^{\infty} \left\vert\frac{\Gamma(d+1)\Gamma(\alpha+n)}{\Gamma(\alpha)\Gamma(d+1+n)}\right\vert \sum_{e_{n,k}\in P_{n}}\langle f,e_{n,k}\rangle e_{n,k} .\end{equation}

 By the direct calculation, we have $$\frac{\vert\mu_{n+1}\vert}{\vert\mu_n\vert}=\frac{\Gamma(\alpha+n+1)}{\Gamma(\alpha+n)}\frac{\Gamma(d+n+1)}{\Gamma(d+n+2)}=\frac{\alpha+n}{d+1+n}<1$$ for each $n>\vert \alpha\vert,$ which means that \begin{equation}\label{muoo}\vert \mu_{n_0}\vert >\vert \mu_{n_0+1}\vert >\vert \mu_{n_0+2}\vert >\cdots\end{equation} where $n_0$ is the minimal integer belongs to  $\{n\in\mathbb{N}:n>\vert \alpha\vert\}.$
On the other hand, it follows from (\ref{cdec}) that $\{\vert \mu_n\vert\}$ is the set of  nonzero eigenvalues (counting multiplicities) of $\vert K_\alpha\vert$ and  the multiplicity $m_n$ of the eigenvalue $\vert \mu_n\vert$  is equal to $dim \hspace{0.3mm}E_n,$  combing with (\ref{dimn}) it follows that  $$m_n= \frac{(n+1)_{d-1}}{(d-1)!}, \quad n\in\mathbb{N}.$$ Denote $M_n$ by $$M_n=\sum_{j=0}^n m_j,$$ for $ n\geq0.$ This along the well-known fact that $\sum_{j=1}^kj^n\sim k^{n+1}$ for any $n\in \mathbb{N}$ implies that \begin{equation}\label{wuqg}M_n=\sum_{k=0}^n m_k\sim n^{d}.\end{equation}
 Then  $K_\alpha\in\mathcal{L}^{p,\infty}$ if and only if $\mu_n( \vert K_\alpha\vert)=O(n^{-\frac{1}{p}})$ if and only if $$\mu_{M_n}(\vert K_\alpha\vert)=O(M_n^{-\frac{1}{p}})$$ if and only if $$n^{\alpha-(d+1)}=O(n^{-\frac{d}{p}})$$
 if and only if $p\geq\frac{d}{d+1-\alpha}.$ It completes the proof. \qed

  \begin{rem} Proposition \ref{fink} follows that $K_\alpha$ is a finite operator whenever  $\alpha$ is a nonpositive integer. In this case,  $K_\alpha\in \mathcal{L}^p(L^2)$ for any $p>0$ rather than $p>\frac{d}{d+1-\alpha}$ and $K_\alpha\in \mathcal{L}^{p,\infty}(L^2)$ for any $p>0$ rather than $p\geq\frac{d}{d+1-\alpha}.$  This shows that the condition $\alpha\in\mathscr{S}_d$ in Theorem 3,4 is necessary and sharp.
\end{rem}

  \begin{cor} For $0\leq\alpha<d+1,$ $K_\alpha$ is bounded on $L^p$ and $A^p$ for any $1<p<\infty.$ Moreover, the following operator norm identity holds.
 $$\Vert K_\alpha\Vert_{L^p\rightarrow L^p}=\Vert K_\alpha\Vert_{A^p\rightarrow A^p}=1.$$
 \end{cor}
 \begin{proof} It comes from Theorem 2, Lemma \ref{klp}, formula (\ref{kass}) and (\ref{muoo}).
 \end{proof}
 Now we recall the the Definition of Dixmier trace on $\mathcal{L}^{1,\infty}(H).$ Let $0\leq T\in\mathcal{L}^{1,\infty}(H)$ and $\{\mu_n(T)\}$ be the sequence of  its point spectrum (counting multiplicities) arranged in decreasing order, namely $$\mu_0(T)\geq\mu_1(T)\geq\mu_2(T)\geq\cdots.$$ Since $0\leq T\in\mathcal{L}^{1,\infty},$ it follows that the sequence $\{\frac{1}{\log n}\sigma_n\}_{n>1}\in \ell^\infty,$ where $\sigma_n=\sum_{j=0}^n\mu_j(T)=\sum_{j=0}^n\mu_j(\vert T\vert).$ Taking an  arbitrary continuous linear functional $\textrm{Lim}_{\omega}$ on $\ell^{\infty}$ satisfying the following three conditions:
\begin{itemize}
   \item [(1)]  $\textrm{Lim}_{\omega}\{a_n\}\geq0$ if $a_n\geq0;$
   \item [(2)]  $\textrm{Lim}_{\omega}\{a_n\}=\textrm{lim}\{a_n\}$ if $\{a_n\}$ is convergent;
   \item [(3)]   $\textrm{Lim}_{\omega}\{a_1,a_1,a_2,a_2,a_3,a_3,\cdots\}=\textrm{Lim}_{\omega}\{a_1,a_2,a_3,\cdots\}.$
 \end{itemize}
 Then the Dixmier trace $\textrm{Tr}_\omega(T)$ is defined by $$\textrm{Tr}_\omega(T)=\textrm{Lim}_{\omega}\frac{1}{\log n}\sigma_n. $$ In this way we define the  Dixmier trace for all positive operators in  $\mathcal{L}^{1,\infty}(H),$ and it can be  uniquely  extended by linearity to the  whole $\mathcal{L}^{1,\infty}(H).$ In general, the Dixmier trace $\textrm{Tr}_\omega$ depends on the linear functional $\textrm{Lim}_{\omega}.$ However, in some special case,  $\textrm{Tr}_\omega$ will be independent of  $\textrm{Lim}_{\omega},$ for example  $\textrm{Tr}_\omega(T)=0$ for any $T\in\mathcal{L}^1.$ We refer the reader to \cite[Chapter 4]{Con} for more details. For positive compact operator  $T\notin\mathcal{L}^{1,\infty}(H),$ we define formally $\textrm{Tr}_\omega( T)=\infty.$ Thus $T\in\mathcal{L}^{p,\infty}(H)$ if and only if  $\vert T\vert^p\in\mathcal{L}^{1,\infty}(H)$ if and only if $\textrm{Tr}_\omega(\vert T\vert^p)<\infty.$ Then the following lemma is trivial.
 \begin{lem}\label{ttrr} Suppose that $T\in\mathcal{L}^{p,\infty}(H)$ for some $0<p<\infty$ and $0<s<t<\infty.$ 
 \begin{itemize}
   \item [(1)] If  $\textrm{Tr}_\omega(\vert T\vert^s)<\infty,$ then $\textrm{Tr}_\omega(\vert T\vert^t)=0.$
   \item [(2)]  If  $\textrm{Tr}_\omega(\vert T\vert^t)=\infty,$ then $\textrm{Tr}_\omega(\vert T\vert^s)=\infty.$
 \end{itemize}
 \end{lem}
 Lemma \ref{ttrr} permits us to introduce the Hausdorff dimension of compact operators as below.  Compare to the Hausdorff dimension for subsets in metric space \cite[Definition 1.1.13]{LiY}.
 \begin{defi}\label{dimhau} The Hausdorff dimension of a compact operator $T$ is defined to be
 $$dim_{\textrm{H}_\omega}(T):=\inf\{0< s<\infty:\textrm{Tr}_\omega(\vert T\vert^s)=0\} ,$$ if there exists $0<p<\infty$ such that  $T\in\mathcal{L}^{p,\infty}(H).$  Otherwise, the Hausdorff dimension of the compact operator $T$ is $$dim_{\textrm{H}_\omega}(T):=\infty.$$
 \end{defi}
  It is immediate that $dim_{\textrm{H}_\omega}(T)\leq p$ if $T\in\mathcal{L}^{p,\infty}(H).$ In fact, there is no difficulty in proving that \begin{equation}\label{hsdf}dim_{\textrm{H}_\omega}(T)=\inf\{0<p<\infty:T\in\mathcal{L}^{p,\infty}(H)\}\end{equation} if $T\in\mathcal{L}^{p,\infty}(H)$ for some $0<p<\infty.$ Which means the definition of Hausdorff dimension of  operator is independent of the  choice of the linear functional $\textrm{Lim}_\omega.$ Thus,  we will omit the functional $\omega$ in $dim_{\textrm{H}_\omega}$ in the following.  In particular,  $dim_{\textrm{H}}(T)=0$ for any finite rank operator $T.$  The following example shows that the Hausdorff dimension can take every value in the interval $[0,\infty].$
\begin{exam}  Let $T$ be a positive compact operator and $\{\mu_n(T)\}$ be the sequence of  its point spectrum (counting multiplicities) arranged in decreasing order.
\begin{itemize}
   \item [(1)] If  $\mu_n(T)=0$ for any $n>n_0,$ then  $dim_{\textrm{H}}(T)=0.$
   \item [(2)] If $\mu_n(T)=n^{-\frac{1}{\beta}}$ for $0<\beta<\infty,$ then $dim_{\textrm{H}}(T)=\beta.$
   \item [(3)] If $\mu_n(T)=\frac{1}{\log n}$,  then $dim_{\textrm{H}}(T)=\infty.$
    \end{itemize}
  \end{exam}
 In the rest of the paper, we will use $\mathcal{L}^{p}$ instead of $\mathcal{L}^{p}(L^2)$ or $\mathcal{L}^{p}(A^2)$ and  $\mathcal{L}^{p,\infty}$ instead of $\mathcal{L}^{p,\infty}(L^2)$ or $\mathcal{L}^{p,\infty}(A^2)$ for $0<p<\infty.$
 \begin{prop}\label{dixtr} Bergman-type operator $ K_\alpha\in\mathcal{L}^{1,\infty}$ if and only if $\alpha\leq1.$ Moreover, the Dixmier trace of $K_\alpha$ is
 \begin{equation}\label{dtra}
\textrm{Tr}_\omega(K_\alpha)=
\begin{cases}
0,& \alpha<1;\\
1,& \alpha=1.
\end{cases}
\end{equation}
   \end{prop} 
 \begin{proof} By Theorem 4, it suffices to prove the Dixmier trace formula (\ref{dtra}). It follows from Theorem 3 that $ K_\alpha\in\mathcal{L}^{1}$ whenever $\alpha<1,$ thus $\textrm{Tr}_\omega(K_\alpha)=0$ whenever $\alpha<1.$ In what follows, we calculate the value of  $\textrm{Tr}_\omega(K_1).$ Note that $K_1$ is positive and its point spectrum is $\{\mu_n=\frac{\Gamma(d+1)\Gamma(1+n)}{\Gamma(d+1+n)}:n\in\mathbb{N}\}$ from Lemma \ref{spec}. By the direct calculation, we have
 $$\frac{\mu_{n+1}}{\mu_n}=\frac{\Gamma(1+n+1)}{\Gamma(1+n)}\frac{\Gamma(d+n+1)}{\Gamma(d+n+2)}=\frac{1+n}{d+1+n}<1,$$  for any $n\geq0,$ namely
 \begin{equation}\label{sppm} \mu_0>\mu_1>\mu_2>\cdots>0.\end{equation} It follows from (\ref{cheq}) and (\ref{dimn}) that the multiplicity $m_n$ of $\mu_n$ is $$m_n= \frac{(n+1)_{d-1}}{(d-1)!}>0$$ for any $n\geq0.$ Note that for any integer $ k\geq0,$ then there exists an unique integer $n\geq0$ satisfying $$M_n\leq k<M_{n+1},$$  where $$M_n=\sum_{j=0}^n m_j,$$ for $ n\geq0.$ Combing this with (\ref{sppm}), it follows that \begin{equation}\label{sigg} \frac{\sigma_{M_n}}{\log M_{n+1}}\leq \frac{1}{\log k}\sigma_k\leq\frac{\sigma_{M_{n+1}}}{\log M_{n}},\end{equation} where $\sigma_k=\sum_{j=0}\mu_j( K_1)=\sum_{j=0}\mu_j( \vert K_1\vert).$ It implies from (\ref{wuqg}) that
  $$\log M_n=\log \sum_{j=0}^n\frac{(j+1)_{d-1}}{(d-1)!}\approx d\log n,$$
 where $a(n)\approx b(n)$ means that $a(n),b(n)\rightarrow +\infty$ and $\frac{a(n)}{b(n)}\rightarrow 1.$ The relationship between $\{\mu_n(K_1)\}$ and $\{\mu_n\}$ implies that
 $$\sigma_{M_n}=\sum_{j=0}^n\mu_jm_j=\sum \frac{\Gamma(d+1)\Gamma(1+j)}{\Gamma(d+1+j)}\frac{(j+1)_{d-1}}{(d-1)!}=\sum_{j=0}^n\frac{d}{j+d}\approx d\log n.$$ Combing with (\ref{sigg}) follows that $$\lim \frac{1}{\log n}\sigma_n=1.$$ Thus $$\textrm{Tr}_\omega(K_1)=\textrm{Lim}_{\omega} \hspace{0.3mm}\frac{1}{\log n}\sigma_n=\lim \frac{1}{\log n}\sigma_n=1,$$ by the condition (2) in the definition of  Dixmier trace.
 It completes the proof.
 \end{proof}
\begin{cor}\label{Hfdim} The followings hold.
\begin{itemize}
   \item [(1)] If  $\alpha\in\mathscr{S}_d,$ then ${dim}_{\textrm{H}}(K_\alpha)=\frac{d}{d+1-\alpha}.$
   \item [(2)]  If  $\alpha$ is a nonpositive integer, then ${dim}_{\textrm{H}}(K_\alpha)=0.$
 \end{itemize}
\end{cor}
\begin{proof} (1) It comes from Theorem 3 and (\ref{hsdf}).

(2) It comes from Lemma \ref{fink} that $K_\alpha$ is finite rank if $\alpha$ is nonpositive integer.
\end{proof}

Then Theorem 3,4, Proposition \ref{fink} and Corollary \ref{Hfdim} imply the following  intrinsic characterization for Schatten class and Macaev class Bergman-type operator $K_\alpha.$ 
\begin{cor}\label{smcha} For $0<p<\infty,$ the followings hold.
\begin{itemize}
   \item [(1)] If $K_\alpha$ is compact, then $K_\alpha\in \mathcal{L}^{p}$ if and only if $p>{dim}_{\textrm{H}}(K_\alpha).$
   \item [(2)] If  $K_\alpha$ is compact, then $K_\alpha\in \mathcal{L}^{p,\infty}$ if and only if $p\geq{dim}_{\textrm{H}}(K_\alpha).$
 \end{itemize}
 
 The compact condition in Corollary \ref{smcha} can be removed, for example, since both $K_\alpha\in \mathcal{L}^{p}$ and $p>{dim}_{\textrm{H}}(K_\alpha) $ can derive the compactness of $K_\alpha.$ Corollary \ref{smcha} provides an example that  the Schatten membership of a compact operator can be characterized by its Hausdorff dimension. However, compare to \cite[Proposition 4.3.14]{Con}, which  means that  the Schatten membership of some compact operators can be characterized by the Hausdorff dimension of some subsets in metric spaces.
\end{cor}

{\noindent{\bf{Acknowledgements.}}
 The authors would  like to thank  Qi'an Guan and Kai Wang for their helpful  discussions.

\bibliographystyle{plain}
 
\end{document}